\documentclass[12pt]{amsart}
\usepackage{amsthm,amssymb,mathrsfs,mathtools,empheq}
\usepackage[shortlabels]{enumitem}
\usepackage[T1]{fontenc}

\usepackage[notref,notcite,final]{showkeys}
\mathtoolsset{showonlyrefs}
\usepackage{color,stmaryrd,bbm}
\usepackage{geometry}

\newtheorem{mainthm}{Theorem}
\newtheorem{theorem}{Theorem}[section]
\newtheorem*{theorem*}{Theorem}
\newtheorem{corollary}[theorem]{Corollary}

\newtheorem{lemma}[theorem]{Lemma}
\newtheorem{proposition}[theorem]{Proposition}
\newtheorem*{proposition*}{Proposition}

\newtheorem*{conjecture*}{Conjecture}

\theoremstyle{definition}
\newtheorem{definition}[theorem]{Definition}
\newtheorem{remark}[theorem]{Remark}

\numberwithin{equation}{section}

\def\bN {\mathbb{N}}
\def\bP {\mathbb{P}}

\def\bR {\mathbb{R}}
\def\bS {\mathbb{S}}

\def\bZ {\mathbb{Z}}

\def\cB {\mathcal{B}}

\def\cE {\mathcal{E}}

\def\cI {\mathcal{I}}

\def\cM {\mathcal{M}}

\def\cR {\mathcal{R}}

\def\cY {\mathcal{Y}}

\def\scrL{\mathscr{L}}

\newcommand{\tx}[1]{\mathrm{#1}}

\newcommand{\wt}[1]{\widetilde{#1}}

\newcommand{\floor}[1]{\lfloor{#1}\rfloor}
\newcommand{\sh}[1]{#1^\sharp}

\newcommand{\eee}{\mathrm e}

\newcommand{\ud}{\mathrm{\,d}}
\newcommand{\vd}{\mathrm{d}}

\newcommand{\dd}[1]{{\frac{\vd}{\vd{#1}}}}

\title[Statistical mechanics of the wave maps equation in dimension $1+1$]{Statistical mechanics of the wave maps equation \\ in dimension $1+1$}

\author{Zdzis{\l}aw Brze\'zniak}
\address{Department of Mathematics, The University of York, Heslington, York, YO105DD,
UK}
\email{zdzislaw.brzezniak@york.ac.uk}

\author{Jacek Jendrej}
\address{CNRS \& LAGA, Universit\'e Sorbonne Paris Nord, 99 av Jean-Baptiste Cl\'ement, 93430 Villetaneuse, France}
\email{jendrej@math.univ-paris13.fr}

\begin{document}

\begin{abstract}
We study wave maps with values in $\bS^d$, defined on the future light cone $\{|x| \leq t\} \subset \bR^{1+1}$,
with prescribed data at the boundary $\{|x| = t\}$.
Based on the work of Keel and Tao, we prove that the problem is well-posed for locally absolutely continuous boundary data.
We design a discrete version of the problem and prove that for every absolutely continuous boundary data,
the sequence of solutions of the discretised problem converges to the corresponding continuous wave map
as the mesh size tends to $0$.

Next, we consider the boundary data given by the $\bS^d$-valued Brownian motion.
We prove that the sequence of solutions of the discretised problems has an accumulation point
for the topology of locally uniform convergence.
We argue that the resulting random field can be interpreted as the wave-map evolution
corresponding to the initial data given by the Gibbs distribution.
\end{abstract}


\thanks{J. Jendrej is supported by ANR-18-CE40-0028 project ESSED. He thanks Charles Collot
for discussions about evolution PDEs with random initial data.}

\maketitle
\section{Introduction}
\label{sec:intro}

\subsection{Setting of the problem}
\label{ssec:setting}

We say that a smooth map $\psi: \bR^{1+1} \supset \Pi \to \bS^d \subset \bR^{d+1}$ is a \emph{wave map} on an open subset of space-time $\Pi$ if there exists a function $\mu: \Pi \to \bR$ such that
\begin{equation}
\label{eq:wm}
\partial_t^2 \psi(t, x) - \partial_x^2\psi(t, x) = \mu(t, x)\psi(t, x), \qquad\text{for all }(t, x) \in \Pi.
\end{equation}
This equation is obtained as the Euler-Lagrange equation for the action functional
\begin{equation}
\scrL(\psi, \partial_t \psi) := \iint \Big(\frac 12 |\partial_t \psi|^2 - \frac 12 |\partial_x \psi|^2\Big)\ud x \ud t
\end{equation}
under the constraints $\psi(t, x) \in \bS^d$.
Since the Lagrangian has exactly the same form as in the case of the usual wave equation,
\eqref{eq:wm} is a natural analog of the linear wave equation in the nonlinear, geometric setting.
For a discussion of the initial value problem for \eqref{eq:wm},
in particular a proof of global well-posedness for initial data of finite energy, we refer to \cite[Chapter 7]{ShSt00}.

Our goal is to study the solutions of \eqref{eq:wm} in the context of Statistical Mechanics.
Recall that to a Hamiltonian system in a phase space of finite dimension
\begin{equation}
\dot q = \partial_p H, \qquad \dot p = -\partial_q H, \qquad (q, p) \in \cM,
\end{equation}
one associates the Gibbs distribution at temperature $T = 1/(k_B\beta)$,
which is a probability distribution on $\cM$ whose density is
\begin{equation}
\rho_T(q, p) := \frac{1}{Z_T}\exp(-\beta H(q, p)),
\end{equation}
where $k_B$ is the Boltzmann constant and $Z_T := \int \exp(-\beta H(q, p))\ud q\ud p$ is the normalising constant.
It is the probability distribution of the position and momentum in the \emph{canonical ensemble},
modelising a system of a large number of non-interacting particles,
in contact with a heat reservoir at temperature $T$.

By analogy, we can consider the infinite-dimensional Hamiltonian system \eqref{eq:wm}.
Proceeding in a non-rigorous way, the Hamiltonian is given by
\begin{equation}
H(\psi, \partial_t \psi) = \partial_t\psi\frac{\partial \scrL}{\partial(\partial_t\psi)} - \scrL = \int_{-\infty}^\infty \Big(\frac 12 |\partial_t \psi|^2 + \frac 12 |\partial_x \psi|^2\Big)\ud x.
\end{equation}
Hence, keeping in mind that $\psi$ has values in $\bS^d$, the Gibbs distribution should be proportional to
\begin{equation}
\label{eq:gibbs}
\begin{aligned}
(\psi, \partial_t\psi) &\mapsto \exp\bigg({-}\beta\int_{-\infty}^\infty \Big(\frac 12 |\partial_t \psi|^2 + \frac 12 |\partial_x \psi|^2\Big)\ud x\bigg) \\
&= \exp\bigg({-}\beta\int_{-\infty}^\infty \frac 12 |\partial_t \psi|^2 \ud x\bigg)\exp\bigg({-}\beta\int_{-\infty}^\infty \frac 12 |\partial_x \psi|^2 \ud x\bigg).
\end{aligned}
\end{equation}
Thus, $\psi$ is a sphere-valued Brownian motion, and $\partial_t \psi$ is a white noise along $\psi$,
meaning that at each $x$, $\partial_t \psi(x)$ is a normally distributed vector in the plane tangent to $\bS^d$ at $\psi(x)$,
and $\partial_t\psi(x)$ is independent of $\partial_t \psi(y)$ whenever $x \neq y$.

The questions which arise are:
\begin{itemize}
\item Can one rigorously define the Gibbs distribution?
\item Is it possible to solve \eqref{eq:wm} for initial data given by the Gibbs distribution?
\end{itemize}

\subsection{Main results}
\label{ssec:results}

Our main goal is to contribute to the second question above.
We avoid dealing with the first question by adopting the perspective of McKean~\cite{McKean-CMP},
which is to construct a random field $\bR^{1+1} \to \bS^d$,
whose marginal distributions at any fixed time correspond to the Gibbs distribution,
and which is a solution of the equation in a sense to be made precise.
We rigorously construct such a random field, and then non-rigorously argue that
the marginal distributions are as expected.

McKean's approach can be contrasted with Bourgain's ``individual trajectory'' approach
of constructing a \emph{deterministic} evolution for all initial data in a set of full measure,
see the comparison of the two viewpoints in \cite[Section 1]{McKean-CMP}.

It is convenient to introduce the \emph{null coordinates} $(u, v) = (t + x, t - x)$.
If we set $\phi(u, v) := \psi(t, x)$, then equation \eqref{eq:wm} reads
\begin{equation}
\label{eq:wm-uv}
\partial_u\partial_v \phi(u, v) = \lambda(u, v)\phi(u, v), \qquad\text{for all }(u, v) \in \Sigma,
\end{equation}
with $\Sigma$ being the image of $\Pi$ under the change of coordinates and $\lambda(u, v) := \frac 14 \mu(t, x)$.
Applying $\partial_u\partial_v$ to both sides of the relation $\phi\cdot\phi = 1$, we have
$\partial_u\partial_v \phi\cdot \phi = -\partial_u \phi\cdot \partial_v \phi$, thus \eqref{eq:wm-uv}
can be rewritten as
\begin{equation}
\label{eq:wm-null}
\partial_u\partial_v\phi(u, v) = -(\partial_u\phi(u, v)\cdot\partial_v\phi(u, v))\phi(u, v).
\end{equation}

We are interested in the
\emph{characteristic Cauchy problem}
\begin{equation}
\label{eq:wm-cauchy}
\begin{gathered}
\partial_u\partial_v \phi(u, v) = -(\partial_u\phi(u, v)\cdot\partial_v\phi(u, v))\phi(u, v), \qquad\text{for all }(u, v) \in (0, \infty)^2, \\
\phi(u, 0) = \phi_+(u), \quad \phi(0, v) = \phi_-(v), \qquad\text{for all }u, v \in [0, \infty),
\end{gathered}
\end{equation}
where the boundary data $\phi_\pm$ (which will always be assumed to be continuous) satisfy $\phi_-(0) = \phi_+(0)$.
The characteristic Cauchy problem can be formulated in the same way in any rectangle
$[u_1, u_2] \times [v_1, v_2] \subset \bR^2$.

It is not difficult to guess, see Section~\ref{ssec:gibbs}, that the boundary data corresponding to the initial data having the Gibbs distribution,
should be given by two independent Brownian motions starting from the same, randomly chosen point on $\bS^d$.
Thus, let us temporarily forget about the Gibbs distribution and ask about the possibility of solving
\eqref{eq:wm-cauchy} for $\phi_+$ and $\phi_-$ being two independent Brownian motions with $\phi_+(0) = \phi_-(0)$
chosen uniformly on $\bS^d$.

At the level of regularity $C^{\frac 12-}$, it is unclear how to specify the notion of solution of \eqref{eq:wm-cauchy}.
Our approach is to define a solution of \eqref{eq:wm-cauchy} as any locally uniform limit of a sequence of solutions of certain \emph{regularisations} of \eqref{eq:wm-cauchy}.
These regularised problems are in fact given by a specific \emph{numerical scheme} for solving \eqref{eq:wm-cauchy},
with the mesh size converging to $0$. We call $\Phi_N(\phi_+, \phi_-)$ the solution of the regularised problem
with mesh size $2^{-N}$, and boundary data $\phi_+$ and $\phi_-$, see Definition~\ref{def:phiN-def}.
The main result of Section~\ref{sec:random} can be stated as follows, see Theorem~\ref{thm:main} for a rigorous formulation.
\begin{mainthm}
\label{thm:1}
Let $\phi_+$ and $\phi_-$ be independent Brownian motions with $\phi_+(0) = \phi_-(0)$
chosen uniformly on $\bS^d$. There exists a sequence $N_k \to \infty$ such that $\Phi_{N_k}(\phi_+, \phi_-)$
converges in distribution to a continuous random field $\phi: [0, \infty)^2 \to \bS^d$.
This random field $\phi$ is invariant by time-like translations of the space-time and has the same modulus of continuity as the Brownian motion.
\end{mainthm}
We also provide a non-rigorous argument that the marginal distribution for fixed time is the Gibbs distribution,
see Section~\ref{ssec:gibbs}.

The purpose of Section~\ref{sec:discrete} is to justify our notion of solution of \eqref{eq:wm-cauchy}.
By the work of Keel and Tao~\cite{Keel+Tao_1998}, the problem \eqref{eq:wm-cauchy}
is well-posed for absolutely continuous data.
We prove that in this case, our notion of solution coincides with the one obtained from \cite{Keel+Tao_1998}.
\begin{mainthm}
\label{thm:2}
If $\phi_+, \phi_-: [0, \infty) \to \bS^d$ are locally absolutely continuous with $\phi_+(0) = \phi_-(0)$,
then the sequence $\Phi_N(\phi_+, \phi_-)$ converges locally uniformly to the solution of \eqref{eq:wm-cauchy}.
\end{mainthm}
Sections~\ref{sec:discrete} and \ref{sec:random} are logically independent, since the Brownian motion
is not absolutely continuous.
In Section~\ref{sec:discrete} we hope to convince the reader that our notion
of solution is reasonable, but the obtained results are not directly used in Section~\ref{sec:random}.

\subsection{Related works}
\label{ssec:biblio}
The study of Gibbs measures for infinite dimensional Hamiltonian systems was initiated
by Lebowitz, Rose and Speer~\cite{LRS-88} in the context of nonlinear Schr\"odinger equations.
Bourgain~\cite{Bourgain-94} consider the question of almost sure (with respect to the Gibbs measure)
global well-posedness for the nonlinear Schr\"odinger equation on the circle,
introducing the idea that the invariance of the Gibbs measures can play a similar role as a conservation law
in a proof of global well-posedness.
McKean and Vaninsky \cite{McKean-Vaninsky-94}, and Zhidkov~\cite{Zhidkov-94} constructed the Gibbs measures
and proved their invariance for certain nonlinear wave equations.

In the recent years, extensive literature was devoted to the use of probabilistic techniques,
most often combined with modern Fourier Analysis,
in the study of the questions of local well-posedness for rough initial data,
as well as global well-posedness for supercritical equations.
Among the notable contributions, we mention the breakthrough results of Burq and Tzvetkov \cite{BT-08-1, BT-08-2}
on supercritical wave equations, the recent theory of random tensors of Deng, Nahmod and Yue~\cite{Deng-22},
global well-posedness results for energy-critical equations by Bringmann~\cite{Bringmann},
Kenig and Mendelson~\cite{KM-19}, Krieger, L{\"u}hrmann and Staffilani~\cite{KLS-20}.
Other important advances include \cite{GKO-18, KMV-20, LM-14, CS-15}, and we would like to refer
the reader to these works for a fuller bibliography.

Recently, Bringmann, L\"uhrmann and Staffilani~\cite{BLS-21}
gave a rigorous definition of the Gibbs distribution for \eqref{eq:wm} and
proved local existence of solutions with probability one (in fact, in the more general setting of any
compact target manifold instead of $\bS^d$).
For the convenience of the reader, we state here Theorem~1.3 from \cite{BLS-21}
(we refer to the original paper for the precise meaning of all the objects playing a role in the statement).
\begin{theorem*}[Bringmann, L{\"u}hrmann and Staffilani, 2021]
Fix $s < \frac 12$. Let $B: \bR \to \cM$ be the Brownian path, let $V: \bR \to T\cM$ be the white noise along $B$,
and let $(B^\epsilon)_{\epsilon > 0}$ and $(V^\epsilon)_{\epsilon > 0}$ be their smooth approximations.
Then, for all $\tau > 0$ and $R\geq 1$, there exists an event $\cE(\tau, R)$ such that the following two properties
hold:
\begin{enumerate}[(i)]
\item  (``High''-probability) We have that
\begin{equation}
\bP(\cE(\tau, R)) \geq 1 - CR\exp({-}c\tau^{-c}),
\end{equation}
where $C = C(\cM)$ and $c > 0$ are constants.
\item (local well-posedness) On the event $\cE(\tau, R)$, the smooth global solutions $\psi^\epsilon$ of \eqref{eq:wm}
with initial data $(B^\epsilon, V^\epsilon)$ converge in $(C_t^0C_x^s \cap C_t^1C_x^{s-1})([{-}\tau, \tau]\times [-R, R] \to \cM)$.
\end{enumerate}
\end{theorem*}
The result above has at least two major advantages with respect to our Theorem~\ref{thm:1}.
First, it provides \emph{uniqueness}, whereas in Theorem~\ref{thm:1} it is a priori possible
to obtain distinct limit fields $\phi$ for various sequences $N_k$.
Second, only the initial data is regularised, but not the equation,
hence the notion of solution used in \cite{BLS-21} is less controversial than ours.
On the other hand, the advantage of Theorem~\ref{thm:1} is to provide globally defined solutions.

In the same paper it was proved, see \cite[Theorem 1.2]{BLS-21}, that $C^{\frac 12+}$ is the threshold regularity
for deterministic local well-posedness of \eqref{eq:wm}, thus any positive results at the level of regularity
of the Brownian motion have to exploit randomness in an essential way.
It is known since the work \cite{BT-08-1} cited above, that randomisation of the initial
data can allow to obtain results on local well-posedness at a lower regularity
than what is achievable in the deterministic setting.

Finally, we mention that in the field of stochastic PDEs, the source of randomness
is most often given by a random forcing term, as opposed to random initial conditions
considered here.
Papers on the stochastic wave maps equation driven by spatially regular Wiener process
include \cite{Brz+Ondr_2007, Brz+Ondr_2013}.

\subsection{Organisation of the paper}
\label{ssec:summary}
In Section~\ref{sec:cauchy}, we build the local well-posedness theory for \eqref{eq:wm-cauchy}
in the case of locally absolutely continuous boundary data.
Even if our arguments are more-less straightforward adaptations of \cite[Section 9]{Keel+Tao_1998}
to the setting of the characteristic Cauchy problem,
we decided to present complete proofs, since they serve as a model
for the analogous results in the case of the discrete wave maps equation, which is done in Section~\ref{sec:discrete}.
The idea to study the characteristic Cauchy problem comes from the mathematical theory of General Relativity, see \cite{ChBr}.
In a similar context as in the present paper, it appears in the work of McKean~\cite{McKean-95}.

At the beginning of Section~\ref{sec:discrete}, we introduce the (deterministic) discrete wave maps equation.
By analogy with the continuous case, we develop a well-posedness theory,
including the crucial ``long time perturbation'', see Lemma~\ref{lem:l-time-pertb},
which quantifies how approximate solutions differ from exact ones.
Theorem~\ref{thm:2} is then proved as follows. We let $\phi$ be a (continuous) wave map defined on some rectangle
of the $(u, v)$-plane. We fix a mesh size (the same for both variables) and evaluate $\phi$ at the mesh points.
We check that the resulting matrix is an approximate solution of the discrete wave maps equation,
and conclude by invoking the long time perturbation for the discrete equation.

Section~\ref{sec:random} is devoted to proving Theorem~\ref{thm:2}.
An evaluation of the boundary data at the mesh points yields a Markov chain
with the transition densities given by the heat kernel.
We then observe that the symmetry property of the heat kernel
leads to the invariance of the distribution of the boundary data by the discrete wave maps evolution, see Lemma~\ref{lem:reflections}.

Already in the early works on the topic, some kind of discretisation of the space variable
was used in order to construct approximations of a given Hamiltonian PDE by a sequence
of Hamiltonian ODEs. For instance, Zhidkov~\cite{Zhidkov-91} employed a discretisation in the physical space,
whereas in \cite{Bourgain-94}, a cut-off in the Fourier space was imposed.
In these works, the invariance of the Gibbs measures for the approximating ODEs is crucially used.
Our Lemma~\ref{lem:reflections} plays an analogous role.
We believe that, in the context of the study of Gibbs distributions,
it is our original contribution to use a space-time discretisation in order to exploit the null structure of the wave maps equation.

By an argument based on L\'evy's proof of his Modulus of Continuity Theorem, see \cite[Section 1.6]{McKean-SI},
the invariance property mentioned above allows us to obtain bounds on the modulus of continuity
of the discrete wave maps when the mesh size converges to $0$.
The final step of our proof relies on the Prokhorov Theorem,
an idea going back, in a similar context, to McKean~\cite{McKean-CMP}, and Burq, Thomann and Tzvetkov~\cite{BTT-18}.


\subsection{Notation}
\label{ssec:notation}
The set of integers is denoted by $\bZ$,
the set of non-negative integers by $\bN$ and the set of positive integers by $\bN^*$.

For
intervals $[u_1, u_2], [v_1, v_2] \subset [0, \infty)$, we denote by $C([u_1, u_2]) = C([u_1, u_2]; \bR^{d+1})$
and by $C([u_1, u_2]\times [v_1, v_2]) = C([u_1, u_2]\times [v_1, v_2]; \bR^{d+1})$ the spaces of
continuous functions with the sup norm.
We denote by $X([u_1, u_2]) \subset C([u_1, u_2])$
and by $X([u_1, u_2]\times [v_1, v_2]) \subset C([u_1, u_2]\times [v_1, v_2])$ the sets of $\bS^d$-valued functions.
These are closed sets.

We denote by $C([0, \infty)) = C([0, \infty); \bR^{d+1})$
and by $C([0, \infty)^2) = C([0, \infty)^2; \bR^{d+1})$ the spaces of
continuous functions with the topology of uniform convergence on compact sets.
These are separable Fr\'echet spaces.
We denote by $X([0, \infty)) \subset C([0, \infty))$
and by $X([0, \infty)^2) \subset C([0, \infty)^2)$ the sets of $\bS^d$-valued functions.
These are closed sets.

If $N \in \bN$ and $(Y(0), Y(1), \ldots, Y(N))$ is a sequence with values in a vector space,
we denote by $EY$ the sequence $(Y(1), Y(2), \ldots, Y(N))$ and by $\delta Y$ the sequence $(Y(1) - Y(0), Y(2) - Y(1), \ldots, Y(N) - Y(N-1))$.
We use the same notation for infinite sequences.
If $(Y(k, n))_{k, n}$ is a double sequence, we set
$E_k Y(k, n) := Y(k+1, n)$, $E_n Y(k, n) := Y(k, n+1)$,
$\delta_k Y(k, n) := Y(k+1, n) - Y(k, n)$
and $\delta_n Y(k, n) := Y(k, n+1) - Y(k, n)$, thus for instance
$\delta_k\delta_n Y(k, n) = Y(k+1, n+1) - Y(k+1, n)-Y(k, n+1) + Y(k, n)$.

For $N_1, N_2 \in \{{-}\infty\} \cup \bZ \cup \{\infty\}$,
we set $\llbracket N_1, N_2 \rrbracket := \{n \in \bZ: N_1 \leq n \leq N_2\}$.
The norms $\ell^1$ and $\ell^\infty$ of sequences (finite or infinite, single or multiple) are defined as usual.

For $x \in \bR$, the unique integer $n$ such that $n \leq x < n+1$ is denoted by $\floor{x}$,
and the fractional part $x - n \in [0, 1)$ is denoted by $\{x\}$.
A number $x \in \bR$ is called \emph{dyadic} if there exist $k \in \bZ$ and $N \in \bN$ such that $x = \frac{k}{2^N}$.

\section{Well-posedness for absolutely continuous boundary data}
\label{sec:cauchy}

\subsection{Notion of a solution}
Let $u_0, v_0 > 0$. For two continuous functions
$\phi_+: [0, u_0] \to \bR^{d+1}$ and $\phi_-: [0, v_0] \to \bR^{d+1}$
such that $\phi_+(0) = \phi_-(0)$, and an $L^1$ function $f: [0, u_0]\times [0, v_0] \to \bR^{d+1}$, the unique solution of the boundary value problem
\begin{equation}
\label{eq:lin-nonhom-wave}
\begin{gathered}
\partial_u\partial_v \phi(u, v) = f(u, v), \quad
\phi(u, 0) = \phi_+(u), \quad \phi(0, v) = \phi_-(v)
\end{gathered}
\end{equation}
is given for all $(u, v) \in [0, u_0]\times [0, v_0]$ by
\begin{equation}
\label{eq:lin-nonhom}
\begin{aligned}
\phi(u, v) := \phi_+(u) + \phi_-(v) - \phi_+(0) + \int_0^u\int_0^v f(w, z)\ud z\ud w.
\end{aligned}
\end{equation}

Motivated by this formula, we have the following notion of solution of \eqref{eq:wm-cauchy}.
\begin{definition}
\label{def:strong}
Let $u_0, v_0 > 0$. We say that a function $\phi:[0, u_0]\times [0, v_0] \to \bS^{d}$
is a solution of the problem \eqref{eq:wm-cauchy} if $\partial_u \phi$ and $\partial_v\phi$
exist almost everywhere, $\partial_u \phi\cdot \partial_v\phi \in L^1([0, u_0]\times [0, v_0])$
and for all $(u, v) \in [0, u_0]\times [0, v_0]$
\begin{equation}
\label{eq:strong}
\phi(u, v) := \phi_+(u) + \phi_-(v) - \phi_+(0) - \int_0^u\int_0^v [(\partial_u\phi\cdot\partial_v\phi)\phi](w, z)\ud z\ud w.
\end{equation}
\end{definition}
\begin{remark}
Note that the definition above implies that
a wave map $\phi$ for continuous boundary data $\phi_\pm$ is continuous.
It also follows from Lemmas~\ref{lem:exist-partial} and \ref{lem:deriv-g+} that
if $\phi$ is a wave map according to the definition above, then $\partial_u\phi, \partial_v\phi$ and $\partial_u\partial_v\phi$ exist almost everywhere,
and \eqref{eq:wm-null} is satisfied almost everywhere.
\end{remark}

It is clear that if $\phi: [0, u_0]\times [0, v_0] \to \bS^d$ is a wave map
with boundary data $\phi_+$ and $\phi_-$, and $(\wt u_0, \wt v_0) \in (0, u_0] \times (0, v_0]$, then $\phi\vert_{[0, \wt u_0]\times [0, \wt v_0]}$ is a wave map with boundary
data $\phi_+\vert_{[0, \wt u_0]}$ and $\phi_-\vert_{[0, \wt v_0]}$.
Conversely, one can glue together wave maps as follows.
\begin{proposition}
\label{prop:gluing}
Let $u_0, v_0 > 0$, $\phi_+ \in X([0, u_0])$, $\phi_- \in X([0, v_0])$
and $\phi: [0, u_0]\times [0, v_0] \to \bS^d$.
Assume that $\wt u_0 \in (0, u_0)$, $\phi\vert_{[0, \wt u_0]\times [0, v_0]}$ is a wave map
with boundary data $\phi_+\vert_{[0, \wt u_0]}$ and $\phi_-$,
and that $\wt \phi: [0, u_0 - \wt u_0]\times [0, v_0]\to \bS^d$ defined by
\begin{equation}
\wt \phi(u, v) := \phi(\wt u_0 + u, v)
\end{equation}
is a wave map with boundary data $u \mapsto \phi_+(\wt u_0 + u)$ and $v \mapsto \phi(\wt u_0, v)$.
Then $\phi$ is a wave map with boundary data $\phi_\pm$.

An analogous result holds if we exchange the roles of $u$ and $v$.
\end{proposition}
\begin{proof}
By assumption, $\partial_u \phi\cdot \partial_v \phi \in L^1([0, \wt u_0] \times [0, v_0])$
and $\partial_u \phi\cdot \partial_v \phi \in L^1([\wt u_0, u_0] \times [0, v_0])$,
thus $\partial_u \phi\cdot \partial_v \phi \in L^1([0, u_0] \times [0, v_0])$.
Let $u \geq \wt u_0$.
By assumption,
\begin{equation}
\begin{aligned}
\phi(u, v) &= \phi_+(u) + \phi(\wt u_0, v) - \phi_+(\wt u_0) - \int_{\wt u_0}^u\int_0^v[(\partial_u\phi\cdot\partial_v\phi)\phi](w, z)\ud z\ud w, \\
\phi(\wt u_0, v) &= \phi_+(\wt u_0) + \phi_-(v) - \phi_+(0) - \int_0^{\wt u_0}\int_0^v[(\partial_u\phi\cdot\partial_v\phi)\phi](w, z)\ud z\ud w.
\end{aligned}
\end{equation}
Replacing $\phi(\wt u_0, v)$ in the first line by the right hand side of the second line,
we obtain \eqref{eq:strong}.
\end{proof}

\subsection{Local and global well-posedness}
We now explain how to adapt \cite[Section 9]{Keel+Tao_1998} to solve the characteristic Cauchy problem.
Let $u_0, v_0 > 0$. For any $\phi_+: [0, u_0] \to \bR^{d+1}$ and $\phi_-: [0, v_0] \to \bR^{d+1}$ such that $\phi_+' \in L^1([0, u_0])$, $\phi_-' \in L^1([0, v_0])$ and $\phi_+(0) = \phi_-(0)$,
and $f: [0, u_0]\times [0, v_0] \to \bR^{d+1}$, $f \in L^1([0, u_0]\times [0, v_0])$, we set
\begin{equation}
\Psi(\phi_+, \phi_-, f) := {-}(\partial_u \phi\cdot\partial_v \phi)\phi,
\end{equation}
where $\phi = \phi(\phi_+, \phi_-, f)$ is defined by \eqref{eq:lin-nonhom}.
The partial derivatives of $\phi$ can be interpreted in the sens of distributions,
but also, by Lemma~\ref{lem:exist-partial}, in the classical sense almost everywhere.

Note that for almost all $(u, v) \in [0, u_0] \times [0, v_0]$
\begin{equation}
\label{eq:Psi-at-0}
[\Psi(\phi_+, \phi_-, 0)](u, v) = -(\phi_+'(u)\cdot \phi_-'(v))(\phi_+(u) + \phi_-(v) - \phi_+(0)).
\end{equation}
\begin{lemma}
\label{lem:trilinear}
There exists $C > 0$ such that the following holds.
If $u_0, v_0 > 0$, $\phi_+, \sh\phi_+: [0, u_0]\to\bR^{d+1}$,
$\phi_-, \sh \phi_-: [0, v_0]\to\bR^{d+1}$ and $f, \sh f: [0, u_0]\times[0, v_0]\to\bR^{d+1}$
are such that
\begin{equation}
\label{eq:trilinear-ass-1}
\phi_+(0) = \phi_-(0),\quad |\phi_+(0)| \leq 2,\quad \sh\phi_+(0) = \sh\phi_-(0),\quad
|\sh \phi_+(0)| \leq 2
\end{equation}
and
\begin{equation}
\label{eq:trilinear-ass-2}
\begin{aligned}
\max\big(\|\phi_+'\|_{L^1}, \|\phi_-'\|_{L^1}, \|f\|_{L^1}, \|(\sh\phi_+)'\|_{L^1}, \|(\sh\phi_-)'\|_{L^1}, \|\sh f\|_{L^1}\big) \leq 1,
\end{aligned}
\end{equation}
then
\begin{equation}
\label{eq:trilinear-1}
\begin{aligned}
&\|\Psi(\sh\phi_+, \sh\phi_-, \sh f) - \Psi(\phi_+, \phi_-, f)\|_{L^1} \leq \\
&\quad C\big((\|\phi_+'\|_{L^1} + \|f\|_{L^1})(\|(\sh \phi_-)' - \phi_-'\|_{L^1} + \|\sh f - f\|_{L^1}) +\\
&\quad+ (\|(\sh \phi_-)'\|_{L^1} + \|\sh f\|_{L^1})(\|(\sh \phi_+)' - \phi_+'\|_{L^1} + \|\sh f - f\|_{L^1}) \\
&\quad+ (\|(\sh \phi_+)'\|_{L^1} + \|\sh f\|_{L^1})(\|(\sh \phi_-)'\|_{L^1} + \|\sh f\|_{L^1}) \times \\
&\qquad \times\big(|\sh\phi_+(0) - \phi_+(0)| + \|(\sh \phi_+)' - \phi_+'\|_{L^1} + \|(\sh \phi_-)' - \phi_-'\|_{L^1} + \|\sh f - f\|_{L^1}\big)\big).
\end{aligned}
\end{equation}
\end{lemma}
\begin{proof}
For $\phi_\pm$ and $f$ as above, we define $g_+ = g_+(\phi_+, f)$ and $g_- = g_-(\phi_+, f)$ by
\begin{align}
\label{eq:g+def}
g_+(u, v) &:= \partial_u \phi(u, v) = \phi_+'(u) + \int_0^v f(u, z)\ud z, \\
\label{eq:g-def}
g_-(u, v) &:= \partial_v \phi(u, v) = \phi_-'(v) + \int_0^u f(w, v)\ud w.
\end{align}
Observe that
\begin{equation}
\sup_{0 \leq v \leq v_0}\int_0^{u_0}|g_+(u, v)|\ud u \leq \int_0^{u_0}|\phi_+'(u)|\ud u +
\int_0^{u_0}\int_0^{v_0}|f(u, v)|\ud v\ud u,
\end{equation}
thus $g_+ \in \scrL(\dot W^{1, 1} \oplus L^1; L_v^\infty L_u^1)$.
We also have
\begin{equation}
\int_0^{v_0}\sup_{0 \leq u \leq u_0}|g_-(u, v)|\ud v \leq \int_0^{v_0}|\phi_-'(v)|\ud v +
\int_0^{v_0}\int_0^{u_0}|f(u, v)|\ud v\ud u,
\end{equation}
thus $g_- \in \scrL(\dot W^{1, 1} \oplus L^1; L_v^1 L_u^\infty)$.
Formula \eqref{eq:lin-nonhom} yields
\begin{equation}
\begin{aligned}
\sup_{0\leq u\leq u_0}\sup_{0\leq v\leq v_0}|\phi(u, v)| &\leq
|\phi_+(0)| + \int_0^{u_0}|\phi_+'(u)|\ud u \\
&+ \int_0^{v_0}|\phi_-'(v)|\ud v + \int_0^{u_0}\int_0^{v_0}|f(u, v)|\ud v\ud u \leq 5,
\end{aligned}
\end{equation}
where the last inequality follows from \eqref{eq:trilinear-ass-1} and \eqref{eq:trilinear-ass-2}.
Similarly,
\begin{equation}
\|\sh \phi - \phi\|_{L^\infty} \leq |\sh\phi_+(0) - \phi_+(0)| + \|(\sh \phi_+)' - \phi_+'\|_{L^1} + \|(\sh \phi_-)' - \phi_-'\|_{L^1} + \|\sh f - f\|_{L^1}.
\end{equation}

For all $\phi \in L_{u, v}^\infty$, $g_+ \in L_v^\infty L_u^1$ and $g_- \in L_v^1 L_u^\infty$,
set $\wt \Psi(\phi, g_+, g_-) := -(g_+\cdot g_-)\phi$.
Since $\wt \Psi: L_{u,v}^\infty \times L_v^\infty L_u^1 \times L_v^1L_u^\infty \to L_{u,v}^1$ is a continuous trilinear functional, we have
\begin{equation}
\begin{aligned}
\|\wt\Psi(\sh \phi, \sh g_+, \sh g_-) - \wt\Psi(\phi, g_+, g_-)\|_{L_{u,v}^1} &\lesssim \|\phi\|_{L_{u,v}^\infty}\|g_+\|_{L_v^\infty L_u^1}\|\sh g_- - g_-\|_{L_v^1 L_u^\infty} \\
&+ \|\phi\|_{L_{u,v}^\infty}\|\sh g_+ - g_+\|_{L_v^\infty L_u^1}\|\sh g_-\|_{L_v^1 L_u^\infty} \\
&+ \|\sh \phi - \phi\|_{L_{u,v}^\infty}\|\sh g_+\|_{L_v^\infty L_u^1}\|\sh g_-\|_{L_v^1 L_u^\infty},
\end{aligned}
\end{equation}
and \eqref{eq:trilinear-1} follows from the fact that
\begin{equation}
\Psi(\phi_+, \phi_-, f) = \wt \Psi\big(\phi(\phi_+, \phi_-, f), g_+(\phi_+, f), g_-(\phi_-, f)\big).
\end{equation}

%
\end{proof}

\begin{theorem}
\label{thm:keel-tao}
If $u_0, v_0 > 0$ and $\phi_+ \in X([0, u_0])$, $\phi_- \in X([0, v_0])$
are such that $\phi_+' \in L^1([0, u_0])$, $\phi_-' \in L^1([0, v_0])$ and $\phi_+(0) = \phi_-(0)$,
then there exists a unique wave map $\phi = \Phi(\phi_+, \phi_-)$ with boundary data $\phi_\pm$.

If $\phi_\pm^{(n)}$ is a sequence of boundary data such that $\phi_+^{(n)}(0) = \phi_-^{(n)}(0)$ and
\begin{equation}
\lim_{n\to\infty}\max(|\phi_+^{(n)}(0) - \phi_+(0)|, \|(\phi_+^{(n)})' - \phi_+'\|_{L^1},
\|(\phi_-^{(n)})' - \phi_-'\|_{L^1}) = 0,
\end{equation}
and $\phi^{(n)}$ are the corresponding wave maps, then
\begin{equation}
\lim_{n\to \infty}\|\partial_u\partial_v (\phi^{(n)} - \phi)\|_{L^1} = 0.
\end{equation}

Moreover, the following ``pointwise conservation laws'' hold:
\begin{align}
\int_0^{u_0}|\partial_u \phi(u, v_0)|\ud u &= \int_0^{u_0}|\phi_+'(u)|\ud u,
\label{eq:pointwise-1} \\
\int_0^{v_0}|\partial_v \phi(u_0, v)|\ud v &= \int_0^{v_0}|\phi_-'(v)|\ud v.
\label{eq:pointwise-2}
\end{align}
\end{theorem}
We first prove that the result holds for small boundary data.
\begin{lemma}
\label{lem:small-data}
There exists $\eta_0 > 0$ such that the
conclusions of Theorem~\ref{thm:keel-tao} hold if $\|\phi_+'\|_{L^1} \leq \eta_0$ and $\|\phi_-'\|_{L^1} \leq \eta_0$.
\end{lemma}
\begin{proof}
\textbf{Step 1: Existence.}
By Lemma~\ref{lem:trilinear}, if we take $R > 0$ sufficiently large
and $\eta_0 > 0$ sufficiently small, then
the map $f \mapsto \Psi(\phi_+, \phi_-, f)$
sends a ball $B(0, 2R\eta_0^2)$ in the space $L^1([0, u_0]\times [0, v_0])$
to the ball $B(0, R\eta_0^2)$ in the same space, and is a strict contraction on $B(0, 2R\eta_0^2)$.
Let $f$ by the unique element of $B(0, 2R\eta_0^2)$ such that $f = \Psi(\phi_+, \phi_-, f)$.
By the definition of a wave map, $\phi = \phi(\phi_+, \phi_-, f)$ given by \eqref{eq:lin-nonhom} is a wave map with boundary data $\phi_\pm$.
The continuity with respect to the boundary data follows from \eqref{eq:trilinear-1}.

\noindent
\textbf{Step 2: Uniqueness.}
Let $\wt \phi$ be a wave map with boundary data $\phi_\pm$
and set $\wt f \in -(\partial_u \wt\phi\cdot \partial_v \wt \phi)\wt \phi
\in L^1([0, u_0]\times [0, v_0])$, so that $\wt f = \Psi(\phi_+, \phi_-, \wt f)$.
It suffices to prove that $\|\wt f\|_{L^1} \leq 2R\eta_0^2$, since
Step 1 will then yield $\wt f = f$.

To this end, suppose by contradiction that $\|\wt f\|_{L^1} > 2R\eta_0^2$ and let $v_1 \in (0, v_0)$
be such that $\|\wt f\|_{L^1(v \leq v_1)} = 2R\eta_0^2$.
Since $\wt f\vert_{v \leq v_1} = \Psi(\phi_+, \phi_-\vert_{v \leq v_1}, \wt f\vert_{v \leq v_1})$, the argument in Step 1 yields $\|\wt f\vert_{v \leq v_1}\|_{L^1} \leq R\eta_0^2$,
contradicting the definition of $v_1$.

\noindent
\textbf{Step 3: Image of the map.}
We now prove that if $f$ is the fixed point, then $\phi = \phi(\phi_+, \phi_-, f)$
takes values in $\bS^d$.
Since $\|f\|_{L^1} \lesssim \eta_0^2$, the proof of Lemma~\ref{lem:trilinear} yields
\begin{equation}
\|\partial_u\phi\cdot\partial_v\phi\|_{L^1} = \|g_+\cdot g_-\|_{L^1} \lesssim \eta_0^2.
\end{equation}
Applying Lemma~\ref{lem:sphere-valued} below with $f := -(\partial_u \phi\cdot\partial_v\phi)\phi$,
we have
\begin{equation}
1 - |\phi(u, v)|^2 = -2\int_0^u\int_0^v (\partial_u \phi(w, z)\cdot\partial_v\phi(w, z))(1 - |\phi(w, z)|^2)\ud z\ud w,
\end{equation}
in particular
\begin{equation}
\|1 - |\phi|^2\|_{L^\infty} \lesssim \eta_0^2 \|1 - |\phi|^2\|_{L^\infty},
\end{equation}
implying $|\phi(u, v)| = 1$ for all $(u, v)$ if $\eta_0$ is sufficiently small.
%

\noindent
\textbf{Step 4: Pointwise conservation laws.}
As in the preceding steps, let $f := -(\partial_u\phi\cdot\partial_v\phi)\phi \in L^1([0, u_0]\times[0, v_0])$. Let $B$ be given by Lemma~\ref{lem:exist-partial} and let $v_1 \in (0, v_0) \setminus B$,
so that for all $u \in [0, u_0]$,
$g(u) := \partial_v\phi(u, v_1)$ exists and is given by \eqref{eq:g-def}.
By the Fubini's Theorem, upon enlarging $B$ by a set of measure $0$, we can also assume
that $\partial_u\phi(u, v_1)$ exists for almost all $u \in [0, u_0]$.

We know by Step 2 that $|\phi(u, v_1)| = 1$ for all $u \in [0, u_0]$,
so the Product Rule yields
\begin{equation}
0 = \frac 12 \partial_v|\phi(u, v_1)|^2 = \phi(u, v_1) \cdot g(u), \qquad\text{for all }u \in [0, u_0].
\end{equation}
Hence we also have
$f(u, v_1) \cdot g(u) = -\big(\partial_u\phi(u, v_1)\cdot\partial_v\phi(u, v_1)\big)\big(\phi(u, v_1)\cdot g(u)\big) = 0$
for almost all $u \in [0, u_0]$.
By Lemma~\ref{lem:preserved-length} with $f(u, v_1)$ instead of $f(u)$, we obtain
\begin{equation}
\Big| \phi_-'(v_1) + \int_0^{u_0}f(w, v_1)\ud w \Big| = |g(u_0)| = |\phi_-'(v_1)|.
\end{equation}
This is true for almost all $v_1$, hence
\begin{equation}
\int_0^{v_0}\Big| \phi_-'(v) + \int_0^{u_0}f(w, v)\ud w \Big|\ud v = \int_0^{v_0}|\phi_-'(v)|\ud v,
\end{equation}
which is \eqref{eq:pointwise-2}, and \eqref{eq:pointwise-1} is proved analogously.
\end{proof}

\begin{remark}
One can check that if the boundary data $\phi_+$, $\phi_-$ are smooth,
then the map $\phi$ constructed
in Step 1 above is smooth as well. In this case, it is straightforward to check
that $\phi$ has values in $\bS^d$ and that the pointwise conservation laws hold.
Approximating general boundary data by a sequence of smooth maps and using
continuous dependence on the boundary data, one would obtain an alternative proof
of Step 3 and Step 4.
\end{remark}

\begin{lemma}
\label{lem:sphere-valued}
Let $u_0, v_0 > 0$, $\phi_+ \in X([0, u_0])$, $\phi_- \in X([0, v_0])$, $f \in L^1([0, u_0]\times [0, v_0])$ and let $\phi: [0, u_0] \times [0, v_0] \to \bR^{d+1}$ be given
by the formula \eqref{eq:lin-nonhom}. Then for all $(u, v) \in [0, u_0]\times [0, v_0]$
\begin{equation}
\label{eq:sphere-valued}
1 - |\phi(u, v)|^2 = -2\int_0^u\int_0^v\big(\partial_u \phi(w, z)\cdot \partial_v \phi(w, z) + \phi(w, z)\cdot f(w, z)\big)\ud z\ud w.
\end{equation}
\end{lemma}
\begin{proof}
By the proof of Lemma~\ref{lem:trilinear}, both sides of \eqref{eq:sphere-valued}
are continuous functions
of $(\phi_+, \phi_-, f) \in \dot W^{1, 1}\times \dot W^{1, 1}\times L^1$,
hence we can assume that $\phi_+, \phi_-, f$ are smooth,
so that $\phi$ is smooth as well. We then obtain
\begin{equation}
\begin{aligned}
\partial_u\partial_v (1 - |\phi(u, v)|^2) &= -2\partial_u \phi(u, v)\cdot\partial_v\phi(u, v) - 2\phi(u, v)\cdot\partial_u\partial_v\phi(u, v) \\
&= -2\big(\partial_u \phi(u, v)\cdot\partial_v\phi(u, v) +\phi(u, v)\cdot f(u, v)\big).
\end{aligned}
\end{equation}
We also have $1 - |\phi(u, 0)|^2 = 0$ for all $u \in [0, u_0]$ and $1 - |\phi(0, v)|^2 = 0$
for all $v \in [0, v_0]$. Hence, by uniqueness, \eqref{eq:lin-nonhom} applied with $1 - |\phi|^2$
instead of $\phi$, yields \eqref{eq:sphere-valued}.
\end{proof}
\begin{proof}[Proof of Theorem~\ref{thm:keel-tao}]
Let $\eta_0$ be given by Lemma~\ref{lem:small-data}.
Let $0 = u_1 < u_2 < \ldots < u_m = u_0$ and $0 = v_1 < v_2 < \ldots < v_n = v_0$
be such that
\begin{align}
\int_{u_j}^{u_{j+1}}|\phi_+'(u)|\ud u &< \eta_0, \qquad\text{for all }j \in \{1, \ldots, m-1\}, \\
\int_{v_k}^{v_{k+1}}|\phi_-'(v)|\ud v &< \eta_0, \qquad\text{for all }k \in \{1, \ldots, n-1\}.
\end{align}
By induction with respect to $i := j + k$, we define $\phi$ as the solution
given by Lemma~\ref{lem:small-data} for boundary data given for
$(u, v) \in [u_j, u_{j+1}] \times \{v_k\} \cup \{u_j\} \times [v_k, v_{k+1}]$.
By \eqref{eq:pointwise-1} and \eqref{eq:pointwise-2}, we have
\begin{equation}
\max\bigg( \int_{u_j}^{u_{j+1}}|\partial_u\phi(v_{k+1}, u)|\ud u, \int_{v_k}^{v_{k+1}}|\partial_v\phi(u_{j+1}, v)|\ud v \bigg) < \eta_0,
\end{equation}
so the smallness of the boundary data is preserved at each step of the induction.
\end{proof}

\section{Discrete wave maps}
\label{sec:discrete}
\subsection{Definition of the numerical scheme}
\label{ssec:numerical}
We first introduce some notation.
For $P, Q \in \bR^{d+1}$, we denote $|P|$ the Euclidean norm of $P$
and $P\cdot Q$ the Euclidean inner product.
Given $Q \in \bR^{d+1} \setminus \{0\}$, we denote
\begin{equation}
\cR_Q: \bR^{d+1} \to \bR^{d+1}, \quad \cR_Q P := 2|Q|^{-2}(Q\cdot P)Q - P,
\end{equation}
the reflection with respect to the line spanned by $Q$.
If $Q = 0$, we set
\begin{equation}
\cR_0: \bR^{d+1} \to \bR^{d+1}, \quad \cR_0 P := -P.
\end{equation}
Note that
\begin{equation}
\label{eq:R-exchange}
\cR_{P + Q}P = Q, \quad \cR_{P + Q}Q = P, \qquad\text{for all }P, Q \in \bS^{d}.
\end{equation}
Observe also that the map
\begin{equation}
(\bS^d)^3 \owns (P, Q, S) \mapsto \cR_{P + Q}S \in \bS^d
\end{equation}
is Borel measurable, by considering separately the cases $P + Q \neq 0$ and $P + Q = 0$.
\begin{remark}
The definition of $\cR_0$ is somewhat arbitrary, and other choices would be possible,
including leaving $\cR_0$ undefined. We made a choice which seemed convenient to us.
\end{remark}

We can now introduce a discrete wave maps equation,
which will directly lead to a numerical scheme for solving the characteristic Cauchy problem.
\begin{definition}
Let $M, N \in \bN \cup \{\infty\}$. We say that $Y : \llbracket 0, M\rrbracket \times \llbracket 0, N\rrbracket \to \bS^d$
is a \emph{discrete wave map} with boundary data $Y_+: \llbracket 0, M\rrbracket \to \bS^d$ and $Y_-: \llbracket 0, N\rrbracket \to \bS^d$ if
\begin{gather}
\label{eq:Y-bdry}
Y(m, 0) = Y_+(m), \quad Y(0, n) = Y_-(n), \qquad\text{for all }m, n, \\
\label{eq:Y-cond}
Y(m+1, n+1) = \cR_{Y(m+1, n) + Y(m, n+1)}Y(m, n), \quad \text{for all }m, n.
\end{gather}
\end{definition}
\begin{lemma}
\label{lem:discr-exist}
If $Y_+: \llbracket 0, M\rrbracket \to \bS^d$ and $Y_-: \llbracket 0, N\rrbracket \to \bS^d$ satisfy $Y_-(0) = Y_+(0)$, then there exists a unique discrete wave map with boundary data $Y_\pm$.
For any $(m, n) \in \llbracket 0, M\rrbracket \times \llbracket 0, N\rrbracket$, the map
\begin{equation}
(Y_+\vert_{\llbracket 0, m\rrbracket}, Y_-\vert_{\llbracket 0, n\rrbracket}) \mapsto Y(m, n)
\end{equation}
is Borel-measurable.
\end{lemma}
\begin{proof}
By induction with respect to $k \in \bN$, conditions \eqref{eq:Y-bdry} and \eqref{eq:Y-cond}
uniquely determine $Y(m, n)$ for all $(m, n)$ such that $m + n \leq k$.

Measurability follows from the fact that a composition of Borel measurable maps is Borel measurable.
\end{proof}
For a given discrete field $Y: \llbracket 0, M\rrbracket \times\llbracket 0, N\rrbracket \to \bS^d$ (respectively
$Y: \bN^2 \to \bS^d$), we define its extension $\wt Y: [0, M]\times [0, N] \to \bR^{d+1}$ (respectively $\wt Y: [0, \infty)^2\to \bR^{d+1}$)
by the formula
\begin{equation}
\label{eq:wtY-def}
\begin{aligned}
\wt Y(u, v) &:= Y(\floor u, \floor v) + \{u\}\delta_k Y(\floor u, \floor v) \\
&+ \{v\}\delta_n Y(\floor u, \floor v)
+ \{u\}\{v\}\delta_k\delta_n Y(\floor u, \floor v) \\
&= (1 - \{u\})(1-\{v\})Y(\floor u, \floor v) \\
&+ \{u\}(1-\{v\})Y(\floor u+1, \floor v) \\
&+ (1-\{u\})\{v\}Y(\floor u, \floor v+1) \\
&+\{u\}\{v\}Y(\floor u+1, \floor v+1),
\end{aligned}
\end{equation}
so that $\wt Y$ agrees with $Y$ at integer points, and on each elementary square, it is affine separately in each variable. In particular, $\wt Y \in C([0, M]\times [0, N])$.
If $u = M$, then $Y(\floor u + 1, v)$ and $Y(\floor u +1, \floor v+1)$ are not defined, but these terms
are multiplied by $\{u\} = 0$, so the formula still makes sense (a similar remark applies to the case $v = N$).
Note that the image of $\wt Y$ is a subset of $\bR^{d+1}$, but in general not a subset of $\bS^d$.
\begin{lemma}
\label{lem:ext-cont}
If $Y: \llbracket 0, M\rrbracket \times\llbracket 0, N\rrbracket \to \bS^d$ and $\wt Y: [0, M]\times [0, N] \to \bR^{d+1}$
is its extension defined above, then
\begin{equation}
\label{eq:ext-lip}
\|\partial_u \wt Y\|_{L^\infty} \leq \|\delta_m Y\|_{\ell^\infty}, \qquad \|\partial_v \wt Y\|_{L^\infty} \leq \|\delta_n Y\|_{\ell^\infty}.
\end{equation}
Moreover, for all $r \geq 2\sqrt 2$
\begin{equation}
\label{eq:ext-cont}
\sup_{|(\sh u - u, \sh v - v)| \leq r}|\wt Y(\sh u, \sh v) - \wt Y(u, v)| \leq \sup_{|\sh m - m| + |\sh n - n| \leq 2r}|Y(\sh m, \sh n) - Y(m, n)|.
\end{equation}
\end{lemma}
\begin{proof}
Let $m_0 \in \llbracket 0, M-1\rrbracket$, $v_0 \in [0, N]$, and let $s \in [0, 1]$ and $n_0 \in \llbracket 0, N-1\rrbracket$
be such that $v_0 = (1-t)n_0 + t(n_0+1)$. Then $\wt Y(m_0, v_0) = (1-s)Y(m_0, n_0) + sY(m_0, n_0+1)$ and $\wt Y(m_0+1, v_0) = (1-s)Y(m_0+1, n_0) + sY(m_0+1, n_0+1)$, thus
\begin{equation}
|\wt Y(m_0+1, v_0) - \wt Y(m_0, v_0)| \leq \|\delta_m Y\|_{\ell^\infty}.
\end{equation}
Since the function $[0, M] \owns u \mapsto \wt Y(u, v_0)$ is affine on every interval $[m_0, m_0+1]$, we obtain the first bound
in \eqref{eq:ext-lip}. The second bound in \eqref{eq:ext-lip} is analogous.

Let $r \geq 2\sqrt 2$, $(u, v), (\sh u, \sh v) \in [0, M]\times [0, N]$ and $|(\sh u - u, \sh v - v)| \leq r$.
Let $(m_j, n_j)_{j=1}^4$ be the vertices of an elementary square containing $(u, v)$,
and $(\sh m_j, \sh n_j)_{j=1}^4$ the vertices of an elementary square containing $(\sh u, \sh v)$.
Since $\wt Y(u, v)$ belongs to the convex hull of
$(Y(m_j, n_j))_{j=1}^4$, and $\wt Y(\sh u, \sh v)$ belongs to the convex hull of $(Y(\sh m_j, \sh n_j))_{j=1}^4$,
there exist $(m, n) = (m_j, n_j)$ and $(\sh m, \sh n) = (\sh m_j, \sh n_j)$ such that
\begin{equation}
|Y(\sh u, \sh v) - Y(u, v)| \leq |Y(\sh m, \sh n) - Y(m, n)|.
\end{equation}
Since
\begin{equation}
|\sh m - m| + |\sh n - n| \leq |\sh u- u| + 2 + |\sh v - v|+2 \leq r\sqrt 2 + 4 \leq 2r,
\end{equation}
the bound \eqref{eq:ext-cont} follows.
\end{proof}

For any boundary data $\phi_\pm \in X([0, \infty))$, we can now define
a sequence of numerical approximations of the corresponding wave map (whose existence is,
for the moment, neither assumed nor claimed).
For $N \in \bN$, the $N$-th approximation is given by the numerical scheme with the mesh size $2^{-N}$.

\begin{definition}
\label{def:phiN-def}
Given $\phi_\pm \in X([0, \infty))$, let $Y: \bN^2 \to \bS^d$ be the unique discrete wave map with the boundary data $Y_\pm(n) := \phi_\pm(2^{-N}n)$, and let $\wt Y$ be its extension \eqref{eq:wtY-def}.
We define $\phi_N = \Phi_N(\phi_-, \phi_+) \in C([0, \infty)^2)$ by the formula
\begin{equation}
\label{eq:phiN-def}
\phi_N(u, v) := \wt Y(2^N u, 2^N v).
\end{equation}
\end{definition}
\begin{remark}
Note that for every $N$ the map $\Phi_N : X([0, \infty)) \times X([0, \infty)) \to
C([0, \infty)^2)$ is Borel.
\end{remark}

\subsection{Convergence for absolutely continuous boundary data}
\label{ssec:scheme-conv}
\begin{theorem}
\label{thm:scheme-conv}
If $\phi_+', \phi_-' \in L^1_\tx{loc}([0, \infty))$, $\phi_+(0) = \phi_-(0)$
and $\phi$ is the solution of \eqref{eq:wm-cauchy} given by Theorem~\ref{thm:keel-tao}, then
\begin{equation}
\label{eq:scheme-conv}
\lim_{N \to \infty} \Phi_N(\phi_+, \phi_-) = \phi\qquad\text{in }C([0, \infty)^2).
\end{equation}
\end{theorem}

In order to prove this theorem, we develop a well-posedness theory for the
characteristic Cauchy problem for discrete wave maps equation with a forcing term.
Our approach is to adapt to the discrete setting the arguments from \cite{Zhou_1999}, and especially \cite[Section 9]{Keel+Tao_1998}.

\begin{definition}
\label{def:sol-Fe-discr}
Let $M, N \in \bN \cup \{\infty\}$.
We say that $Y : \llbracket 0, M \rrbracket \times \llbracket 0, N \rrbracket \to \bS^d$
is a \emph{discrete wave map} with \emph{boundary data} $Y_+: \llbracket 0, M \rrbracket \to \bS^d$,
$Y_-: \llbracket 0, N \rrbracket \to \bS^d$ and \emph{external forcing} $F_\tx e:\llbracket 0, M-1 \rrbracket
\times \llbracket 0, N-1 \rrbracket \to \bR^{d+1}$ if
\begin{gather}
\label{eq:Y-bdry-F}
Y(n, 0) = Y_+(n), \quad Y(0, n) = Y_-(n), \qquad\text{for all }n, \\
\label{eq:Y-cond-F}
Y(m+1, n+1) = \cR_{Y(m+1, n) + Y(m, n+1)}Y(m, n) + F_\tx e(m, n), \qquad\text{for all }m, n.
\end{gather}
\end{definition}
The existence and the uniqueness of solutions for given boundary data and forcing are clear,
but we are interested here in sharp estimates on the sensitivity of the solution
to changing the boundary data or the external forcing.
First, we rewrite the discrete wave maps equation so as to exhibit
the linear and the nonlinear parts.
\begin{lemma}
\label{lem:R-formula}
If $P, Q, S \in \bS^d$ and $P + Q \neq 0$, then
\begin{equation}
\cR_{P+Q}S = P + Q - S - 2[(S - P)\cdot (S - Q)]\cI(P + Q),
\end{equation}
where $\cI(P+Q) := |P+Q|^{-2}(P+Q)$ is the inversion in the unit sphere.
\end{lemma}
\begin{proof}
By the definition of $\cR$, we have
\begin{equation}
\begin{aligned}
\cR_{P+Q}S &= 2\frac{(P+Q)\cdot S}{|P + Q|^2}(P+Q) - S \\
&= \big[2(P+Q)\cdot S - |P+Q|^2\big] \cI(P+Q)  + P + Q - S.
\end{aligned}
\end{equation}
Hence it suffices to observe that
\begin{equation}
2(P+Q)\cdot S - |P + Q|^2 = -2(S - P)\cdot (S - Q),
\end{equation}
since  $P, Q, S \in \bS^d$.
\end{proof}
\begin{lemma}
\label{lem:inversion-lip}
Let $M , N \in \bN$. If $Y, \sh Y \in \ell^\infty(\llbracket 0, M\rrbracket \times \llbracket 0, N\rrbracket)$ are such that
\begin{equation}
\label{eq:modY-12}
\min(|Y(m, n)|, |\sh Y(m, n)|) \geq \frac 12,\qquad\text{for all }(m, n) \in \llbracket 0, M\rrbracket \times \llbracket 0, N\rrbracket
\end{equation}
and
\begin{equation}
\label{eq:Ylinfty-small}
\max\big(\|\delta_m Y\|_{\ell^\infty}, \|\delta_n Y\|_{\ell^\infty}, \|\delta_m \sh Y\|_{\ell^\infty}, \|\delta_n \sh Y\|_{\ell^\infty}\big) \leq \frac{\sqrt 2}{4},
\end{equation}
then
\begin{gather}
\label{eq:inversion-Linfty}
\|\cI((E_m Y + E_n Y)/2)\|_{\ell^\infty} \leq 2\sqrt 2, \\
\label{eq:inversion-lip}
\|\cI((E_m\sh Y + E_n\sh Y)/2) - \cI((E_m Y + E_n Y)/2)\|_{\ell^\infty} \leq
8 \|\sh Y - Y\|_{\ell^\infty}.
\end{gather}
\end{lemma}
\begin{proof}
Fix $(m, n) \in \llbracket 0, M-1\rrbracket \times \llbracket 0, M-1\rrbracket$, and set $P := \frac 12(Y(m+1, n) + Y(m, n+1))$
and $\sh P := \frac 12(\sh Y(m+1, n) + Y(m, n+1))$.
By \eqref{eq:modY-12}, \eqref{eq:Ylinfty-small} and the Parallelogram Law, we have
$|P| \geq \frac{\sqrt 2}{4}$ and $|\sh P| \geq \frac{\sqrt 2}{4}$,
which implies
\begin{equation}
|\cI(P)| = |P|^{-1} \leq 2\sqrt 2.
\end{equation}

Observe that the triangle with vertices $0$, $\cI(P)$
and $\cI(\sh P)$ is similar to the triangle
with vertices $0$, $\sh P$
and $P$, which implies
\begin{equation}
\begin{aligned}
\frac{|\cI(\sh P) - \cI(P)|}{|\sh P - P|} = \frac{|\cI(\sh P)|}{|P|}
= \frac{1}{|P||\sh P|} \leq 8.
\end{aligned}
\end{equation}
By the triangle inequality, $|\sh P - P| \leq 8\|\sh Y - Y\|_{\ell^\infty}$, hence $|\cI(\sh P) - \cI(P)| \leq 8\|\sh Y - Y\|_{\ell^\infty}$.
\end{proof}

Let $M, N \in \bN^*$, $Y_+ : \llbracket 0, M \rrbracket \to \bR^{d+1}$,
$Y_-: \llbracket 0, N \rrbracket \to \bR^{d+1}$ and $F: \llbracket 0, M-1\rrbracket \times \llbracket 0, N-1\rrbracket \to \bR^{d+1}$.
Then the unique solution $Y: \llbracket 0, M\rrbracket \times \llbracket 0, N\rrbracket \to \bR^{d+1}$ of the problem
\begin{gather}
\label{eq:Ylin-bdry-1}
Y(m, 0) = Y_+(m), \qquad\text{for all }m, \\
\label{eq:Ylin-bdry-2}
Y(0, n) = Y_-(n), \qquad\text{for all }n, \\
\label{eq:Ylin-cond}
\delta_m\delta_n Y(m, n) = F(m, n), \qquad\text{for all }m, n
\end{gather}
is given by the formula
\begin{equation}
\label{eq:discr-sol}
Y(m, n) = Y_+(m) + Y_-(n) - Y_+(0) + \sum_{j=0}^{m-1}\sum_{k=0}^{n-1}F(j, k), \qquad\text{for all }m, n.
\end{equation}
Note that
\begin{align}
\delta_m Y(m, n) &= \delta Y_+(m) + \sum_{k=0}^{n-1}F(m, k), \\
\delta_n Y(m, n) &= \delta Y_-(n) + \sum_{j=0}^{m-1}F(j, n),
\end{align}
which are discrete analogs of \eqref{eq:g+def} and \eqref{eq:g-def}.

Given $M, N \in \bN^*$, $Y_+ \in \ell^1(\llbracket 0, M\rrbracket)$ and
$Y_- \in \ell^1(\llbracket 0, N\rrbracket)$ such that $Y_+(0) = Y_-(0)$,
and $F \in \ell^1(\llbracket 0, M-1\rrbracket \times \llbracket 0, N-1\rrbracket)$,
we set
\begin{equation}
\Psi(Y_+, Y_-, F):={-}\delta_m Y\cdot\delta_n Y \cI((E_m Y + E_n Y)/2),
\end{equation}
where $Y = Y(Y_+, Y_-, F)$ is given by \eqref{eq:discr-sol}
and $E$ is the shift operator defined in Section~\ref{ssec:notation}.

\begin{lemma}
\label{lem:trilinear-discr}
There exists $C> 0$ such that the following holds.
Let $M, N \in \bN^*$ and let $\Psi$ be the mapping defined above.
Let $Y_+, \sh Y_+: \llbracket 0, M\rrbracket \to\bR^{d+1}$,
$Y_-, \sh Y_-: \llbracket 0, N\rrbracket\to\bR^{d+1}$ and $F, \sh F: \llbracket 0, M-1\rrbracket\times\llbracket 0, N-1\rrbracket\to\bR^{d+1}$
be such that
\begin{equation}
Y_+(0) = Y_-(0),\quad \frac 78 \leq |Y_+(0)| \leq \frac 98,\quad \sh Y_+(0) = \sh Y_-(0),
\quad \frac 78 \leq |\sh Y_+(0)| \leq \frac 98
\end{equation}
and
\begin{equation}
\label{eq:trilinear-discr-ass}
\begin{aligned}
\max\big(\|\delta Y_+\|_{\ell^1}, \|\delta Y_-\|_{\ell^1}, \|F\|_{\ell^1}, \|\delta\sh Y_+\|_{\ell^1}, \|\delta\sh Y_-\|_{\ell^1}, \|\sh F\|_{\ell^1}\big) \leq \frac 18.
\end{aligned}
\end{equation}
Then
\begin{equation}
\label{eq:trilinear-0-discr}
\|\Psi(Y_+, Y_-, F)\|_{\ell^1} \leq C(\|\delta Y_+\|_{\ell^1} + \|F\|_{\ell^1})
(\|\delta Y_-\|_{\ell^1} + \|F\|_{\ell^1})
\end{equation}
and
\begin{equation}
\label{eq:trilinear-1-discr}
\begin{aligned}
&\|\Psi(\sh Y_+, \sh Y_-, \sh F) - \Psi(Y_+, Y_-, F)\|_{\ell^1} \lesssim \\
&\quad C\big((\|\delta Y_+\|_{\ell^1} + \|F\|_{\ell^1})(\|\delta \sh Y_- - \delta Y_-\|_{\ell^1} + \|\sh F - F\|_{\ell^1}) +\\
&\quad+ (\|\delta \sh Y_-\|_{\ell^1} + \|\sh F\|_{\ell^1})(\|\delta \sh Y_+ - \delta Y_+\|_{\ell^1} + \|\sh F - F\|_{\ell^1}) \\
&\quad+ (\|\delta \sh Y_+\|_{\ell^1} + \|\sh F\|_{\ell^1})(\|\delta Y_-\|_{\ell^1} + \|\sh F\|_{\ell^1}) \times \\
&\qquad \times\big(|\sh Y_+(0) - Y_+(0)| + \|\delta \sh Y_+ - \delta Y_+\|_{\ell^1} + \|\delta\sh Y_- - \delta Y_-\|_{\ell^1} + \|\sh F - F\|_{\ell^1}\big)\big).
\end{aligned}
\end{equation}
\end{lemma}
\begin{proof}
For $Y_\pm$ and $F$ as above, we define $G_+ = G_+(Y_+, F)$, $G_- = G_-(Y_-, F)$
and $Z = Z(Y_+, Y_-, F)$ as follows. For all $(m, n) \in \llbracket 0, M-1\rrbracket
\times \llbracket 0, N-1\rrbracket$, we set
\begin{equation}
\begin{aligned}
G_+(m, n) &:= \delta_m Y(m, n) = \delta Y_+(m) + \sum_{k=0}^{n-1}F(m, k), \\
G_-(m, n) &:= \delta_n Y(m, n) = \delta Y_-(n) + \sum_{j=0}^{m-1}F(j, n), \\
Z(m, n) &:= \cI((E_m Y(m, n) + E_n Y(m, n))/2).
\end{aligned}
\end{equation}
The mappings $(Y_+, F) \mapsto G_+(Y_+, F)$ and $(Y_-, F) \mapsto G_-(Y_-, F)$
are linear and
\begin{align}
\label{eq:G+bd}
\max_{0 \leq n \leq N}\sum_{m=0}^{M-1}|G_+(m, n)| &\leq \sum_{m=0}^{M-1}|\delta Y_+(m)|
+ \sum_{m=0}^{M-1}\sum_{n=0}^{N-1}|F(m, n)|, \\
\label{eq:G-bd}
\sum_{n=0}^{N-1}\max_{0 \leq m \leq M}|G_-(m, n)| &\leq \sum_{m=0}^{N-1}|\delta Y_-(n)|
+ \sum_{m=0}^{M-1}\sum_{n=0}^{N-1}|F(m, n)|.
\end{align}
Observe that the formulas for $\delta_m Y$ and $\delta_n Y$ recalled above also
yield, using \eqref{eq:trilinear-discr-ass},
\begin{equation}
\max(\|\delta_m Y\|_{\ell^\infty}, \|\delta_n Y\|_{\ell^\infty}) \leq \frac 18 + \frac 18 = \frac 14.
\end{equation}
By the formula \eqref{eq:discr-sol}, for all $(m, n) \in \llbracket 0, M\rrbracket
\times \llbracket 0, N\rrbracket$, we have
\begin{equation}
|Y(m, n)| \geq |Y_+(0)| - \|\delta Y_+\|_{\ell^1} - \|\delta Y_-\|_{\ell^1} - \|F\|_{\ell^1}
\geq \frac 78 - 3\times \frac 18 = \frac 12,
\end{equation}
and similarly $|\sh Y(m, n)| \geq \frac 12$.
Hence, from Lemma~\ref{lem:inversion-lip} we deduce that $\|Z\|_{\ell^\infty} \leq 2\sqrt 2$ and
\begin{equation}
\begin{aligned}
&\|\sh Z - Z\|_{\ell^\infty} \leq 8\|\sh Y - Y\|_{\ell^\infty} \leq \\
&\quad \leq 8\big(|\sh Y_+(0) - Y_+(0)| + \|\delta \sh Y_+ - \delta Y_+\|_{\ell^1} + \|\delta\sh Y_- - \delta Y_-\|_{\ell^1} + \|\sh F - F\|_{\ell^1}\big),
\end{aligned}
\end{equation}
where the last inequality follows from \eqref{eq:discr-sol}.

For all $Z \in \ell_{m, n}^\infty$, $G_+ \in \ell_n^\infty \ell_m^1$ and $G_- \in \ell_n^1 \ell_m^\infty$, the domain in each case being $\llbracket 0, M-1\rrbracket
\times \llbracket 0, N-1\rrbracket$,
set $\wt \Psi(Z, G_+, G_-) := -(G_+\cdot G_-)Z$.
Since $\wt \Psi: \ell_{m,n}^\infty \times \ell_n^\infty\ell_m^1 \times \ell_n^1\ell_m^\infty
\to \ell_{m,n}^1$ is a continuous trilinear functional, we have
\begin{equation}
\|\wt\Psi(Z, G_+, G_-)\|_{\ell_{m,n}^1} \lesssim \|Z\|_{\ell_{m,n}^\infty}\|G_+\|_{\ell_n^\infty\ell_m^1}\|G_-\|_{\ell_n^1\ell_m^\infty}
\end{equation}
and
\begin{equation}
\begin{aligned}
\|\wt\Psi(\sh Z, \sh G_+, \sh G_-) - \wt\Psi(Z, G_+, G_-)\|_{\ell_{m,n}^1} &\lesssim \|Z\|_{\ell_{m,n}^\infty}\|G_+\|_{\ell_n^\infty \ell_m^1}\|\sh G_- - G_-\|_{\ell_n^1 \ell_m^\infty} \\
&+ \|Z\|_{\ell_{m,n}^\infty}\|\sh G_+ - G_+\|_{\ell_n^\infty \ell_m^1}\|\sh G_-\|_{\ell_n^1 \ell_m^\infty} \\
&+ \|\sh Z - Z\|_{\ell_{m,n}^\infty}\|\sh G_+\|_{\ell_n^\infty \ell_m^1}\|\sh G_-\|_{\ell_n^1 \ell_m^\infty},
\end{aligned}
\end{equation}
hence \eqref{eq:trilinear-1} follows from the fact that
\begin{equation}
\Psi(Y_+, Y_-, F) = \wt \Psi\big(Z(Y_+, Y_-, F), G_+(Y_+, F), G_-(Y_-, F)\big).
\end{equation}
\end{proof}

\begin{lemma}[Short time perturbation]
\label{lem:s-time-pertb}
There exist $\eta_0, C > 0$ such that the following holds.
Let $M, N \in \bN$, let $Y:\llbracket 0, M\rrbracket\times \llbracket 0, N\rrbracket \to \bS^d$ be a discrete wave map with boundary data $Y_+, Y_-$, and $\sh Y:\llbracket 0, M\rrbracket\times \llbracket 0, N\rrbracket \to \bS^d$ a discrete wave map with boundary data
$\sh Y_+, \sh Y_-$ and external forcing $F_\tx e$, such that
\begin{gather}
\max(\|\delta Y_+\|_{\ell^1}, \|\delta Y_-\|_{\ell^1}, \|\delta \sh Y_+\|_{\ell^1}, \|\delta \sh Y_-\|_{\ell^1}, \|F_\tx e\|_{\ell^1}) \leq \eta_0.
\end{gather}
Then
\begin{equation}
\label{eq:s-time-pertb}
\|\delta_m\delta_n(\sh Y- Y)\|_{\ell^1} \leq 2(|\sh Y_+(0) - Y_+(0)| + \|\sh Y_+ - Y_+\| + \|\sh Y_- - Y_-\| + \|F_\tx e\|).
\end{equation}
\end{lemma}
\begin{proof}
By Lemma~\ref{lem:trilinear-discr}, if $\eta_0$ is small enough, then
the mapping $F \mapsto \Psi(Y_+, Y_-, F)$ is a strict contraction on the ball of center
$0$ and radius $2\eta_0$ in the space $\ell^1(\llbracket 0, M-1\rrbracket
\times \llbracket 0, N-1\rrbracket)$,
thus it has a unique fixed point $F$.
Then $\delta_m \delta_n Y = F$.

Similarly, the mapping $\sh F \mapsto \Psi(\sh Y_+, \sh Y_-, \sh F) + F_\tx e$
has a unique fixed point $\sh F$ in the same ball,
and we see that $\delta_m\delta_n \sh Y = \sh F$.

By Lemma~\ref{lem:trilinear-discr},
\begin{equation}
\begin{aligned}
\|\sh F - F\|_{\ell^1} &\leq \|\Psi(\sh Y_+, \sh Y_-, \sh F) - \Psi(Y_+, Y_-, F)\|_{\ell^1} + \|F_\tx e\|_{\ell^1} \\
&\leq \|\sh Y_+ - Y_+\|_{\ell^1} + \|\sh Y_- - Y_-\|_{\ell^1} + \frac 12 \|\sh F - F\|_{\ell^1} + \|F_\tx e\|_{\ell^1}
\end{aligned}
\end{equation}
if $\eta_0$ is taken small enough.
\end{proof}
\begin{lemma}[Long time perturbation]
\label{lem:l-time-pertb}
Let $\eta_0 > 0$ be given by Lemma~\ref{lem:s-time-pertb}.
For any $A, B > 0$ there exist $C$ and $\epsilon_0$ such that the following holds.
Let $M, N \in \bN$, $0 \leq \epsilon \leq \epsilon_0$.
Let $Y:\llbracket 0, M\rrbracket\times \llbracket 0, N\rrbracket \to \bS^d$ be a discrete wave map with boundary data $Y_+, Y_-$, and $\sh Y:\llbracket 0, M\rrbracket\times \llbracket 0, N\rrbracket \to \bS^d$ a discrete wave map with boundary data
$\sh Y_+, \sh Y_-$ and external forcing $F_\tx e$, such that
\begin{gather}
\label{eq:l-time-pertb-1}
\max(\|\delta Y_+\|_{\ell^\infty}, \|\delta Y_-\|_{\ell^\infty}, \|\delta\sh Y_+\|_{\ell^\infty}, \|\delta\sh Y_-\|_{\ell^\infty}) \leq \frac{\eta_0}{8}, \\
\label{eq:l-time-pertb-2}
\max(\|\delta Y_+\|_{\ell^1}, \|\delta \sh Y_+\|_{\ell^1}) \leq A, \\
\label{eq:l-time-pertb-3}
\max(\|\delta Y_-\|_{\ell^1}, \|\delta \sh Y_-\|_{\ell^1}) \leq B, \\
\label{eq:l-time-pertb-4}
\|\delta (\sh Y_+ - Y_+)\|_{\ell^1} + \|\delta(\sh Y_- - Y_-)\|_{\ell^1} + \|F_\tx e\|_{\ell^1} \leq \epsilon.
\end{gather}
Then
\begin{equation}
\label{eq:l-time-pertb}
\begin{aligned}
&\max\big(\|\delta_m\delta_n(\sh Y- Y)\|_{\ell_{m, n}^1},
\|\delta_n (\sh Y - Y)\|_{\ell_m^\infty\ell_n^1}, \\
&\qquad\qquad\|\delta_m (\sh Y - Y)\|_{\ell_n^\infty\ell_m^1},
\|\sh Y - Y\|_{\ell_{m, n}^\infty}
\big) \leq C\epsilon.
\end{aligned}
\end{equation}
\end{lemma}
\begin{proof}
It suffices to estimate $\|\delta_m\delta_n(\sh Y - Y)\|_{\ell_{m, n}^1}$,
the bounds on the remaining quantities being obtained by summation
with respect to one or both indices, using the formula \eqref{eq:discr-sol}.

\noindent
\textbf{Step 1.} We first prove the lemma under the additional assumption that $A \leq \frac 12\eta_0$. We proceed by induction
with respect to the integer part of $4B\eta_0^{-1}$.
If $B \leq \eta_0$, then the conclusion follows from Lemma~\ref{lem:s-time-pertb},
so it remains to verify the induction step.

Assume $B > \eta_0$.
Let $\wt B := B - \frac 14\eta_0 > \frac 34 \eta_0$.
By the induction hypothesis,
there exists $\wt C$ such that if \eqref{eq:l-time-pertb-3}
holds with $B$ replaced by $\wt B$,
then \eqref{eq:l-time-pertb} holds with $\wt C$ instead of $C$.
We can thus assume that $\|\delta Y_-\|_{\ell^1} > \wt B$ or $\|\delta \sh Y_-\|_{\ell^1} > \wt B$.

Let for example $\|\delta Y_-\|_{\ell^1} > \wt B > \frac 34 \eta_0$
and let $n_0$ be the smallest index such that
\begin{equation}
\label{eq:n_0-cond}
\sum_{n=n_0}^{N-1}|Y_-(n+1) - Y_-(n)| \leq \frac 34 \eta_0.
\end{equation}
We then have $n_0 > 0$ and, since \eqref{eq:l-time-pertb-1} yields $|Y_-(n_0) - Y_-(n_0-1)| \leq \frac 14\eta_0$, we see that
\begin{equation}
\label{eq:n_0-bd}
\sum_{n=n_0}^{N-1}|Y_-(n+1) - Y_-(n)| \in \big(\frac 12 \eta_0, \frac 34\eta_0\big].
\end{equation}
Otherwise, \eqref{eq:n_0-cond} would hold with $n_0$ replaced by $n_0 - 1$,
contradicting the definition of $n_0$.

If we take $\epsilon_0 \leq \frac 14 \eta_0$, then \eqref{eq:n_0-bd} and \eqref{eq:l-time-pertb-4} imply
\begin{equation}
\label{eq:n_0-bd-sh}
\sum_{n=n_0}^{N-1}|\sh Y_-(n+1) - \sh Y_-(n)| \in \big(\frac 14 \eta_0, \eta_0\big],
\end{equation}
thus
\begin{equation}
\max\Big(\sum_{n=0}^{n_0-1}|Y_-(n+1) - Y_-(n)|, \sum_{n=0}^{n_0-1}|\sh Y_-(n+1) - \sh Y_-(n)|\Big) < \wt B.
\end{equation}
By the induction hypthesis, we obtain
\begin{equation}
\label{eq:l-time-pertb-n_0}
\max(\|\delta_m\delta_n(\sh Y\vert_{n \leq n_0}- Y\vert_{n \leq n_0})\|_{\ell^1},
\|\delta (\sh Y(\cdot, n_0) - Y(\cdot, n_0))\|_{\ell^1}) \leq \wt C\epsilon.
\end{equation}
Since $Y$ is a discrete wave map, we also have
\begin{equation}
\|\delta Y(\cdot, n_0)\|_{\ell^1} = \|\delta Y_+\|_{\ell^1} \leq A \leq \frac 12 \eta_0.
\end{equation}
If we take $\epsilon_0 \leq (2\wt C)^{-1}\eta_0$, then \eqref{eq:l-time-pertb-n_0} gives
\begin{equation}
\|\delta \sh Y(\cdot, n_0)\|_{\ell^1} \leq \eta_0.
\end{equation}
Applying Lemma~\ref{lem:s-time-pertb} in the rectangle $\llbracket 0, M\rrbracket \times
\llbracket n_0, N\rrbracket$ finishes the induction step.

\noindent
\textbf{Step 2.} We prove the general case,
proceeding by induction with respect to the integer part of $8A\eta_0^{-1}$.
If $A \leq \frac 12\eta_0$, then the conclusion follows from Step 1,
so it remains to verify the induction step.

Assume $A > \frac 12\eta_0$.
Let $\wt A := A - \frac 18\eta_0 > \frac 38 \eta_0$. By the induction hypothesis,
there exists $\wt C$ such that if \eqref{eq:l-time-pertb-2} holds with $A$ replaced by $\wt A$,
then \eqref{eq:l-time-pertb} holds with $\wt C$ instead of $C$.
We can thus assume that $\|\delta Y_+\|_{\ell^1} > \wt A$ or $\|\delta \sh Y_+\|_{\ell^1} > \wt A$.

Let for example $\|\delta Y_+\|_{\ell^1} > \wt A > \frac 38 \eta_0$
and let $m_0$ be the smallest index such that
\begin{equation}
\label{eq:m_0-cond}
\sum_{m=m_0}^{M-1}|Y_+(m+1) - Y_+(m)| \leq \frac 38 \eta_0.
\end{equation}
We then have $m_0 > 0$ and, since \eqref{eq:l-time-pertb-1} yields $|Y_+(m_0) - Y_+(m_0-1)| \leq \frac 18\eta_0$, we see that
\begin{equation}
\label{eq:m_0-bd}
\sum_{m=m_0}^{M-1}|Y_+(m+1) - Y_+(m)| \in \big(\frac 14 \eta_0, \frac 38\eta_0\big].
\end{equation}
Otherwise, \eqref{eq:m_0-cond} would hold with $m_0$ replaced by $m_0 - 1$,
contradicting the definition of $m_0$.

If we take $\epsilon_0 \leq \frac 18 \eta_0$, then \eqref{eq:m_0-bd} and \eqref{eq:l-time-pertb-4} imply
\begin{equation}
\label{eq:m_0-hash-bd}
\sum_{m=m_0}^{M-1}|\sh Y_+(m+1) - \sh Y_+(m)| \in \big(\frac 18 \eta_0, \frac 12\eta_0\big],
\end{equation}
thus
\begin{equation}
\max\Big(\sum_{m=0}^{m_0-1}|Y_+(m+1) - Y_+(m)|, \sum_{m=0}^{m_0-1}|\sh Y_+(m+1) - \sh Y_+(m)|\Big) < \wt A.
\end{equation}
By the induction hypthesis, we obtain
\begin{equation}
\label{eq:l-time-pertb-m_0}
\|\delta_m\delta_n(\sh Y\vert_{m \leq m_0}- Y\vert_{m \leq m_0})\|_{\ell^1} \leq \wt C\epsilon.
\end{equation}
By Step 1, we also have $\|\delta_m\delta_n(\sh Y\vert_{m \geq m_0}- Y\vert_{m \geq m_0})\|_{\ell^1} \lesssim \epsilon$, and we obtain the induction step by summing the two bounds.
\end{proof}
\begin{remark}
We stress that, in Lemma~\ref{lem:s-time-pertb}, $\eta_0$ and $C$ are independent of $M$ and $N$.
Similarly, in Lemma~\ref{lem:l-time-pertb},
$C$ and $\epsilon_0$ only depend on $A$ and $B$, but not on $M$ and $N$.
\end{remark}

In the next lemma, we compare a solution of the wave maps equation
with small boundary data with its discrete approximation.
\begin{lemma}
\label{lem:cell-sol}
There exist $\eta_0, C > 0$ such that the following holds.
Let $w > 0$ and $\phi: [0, w]\times [0, w] \to \bS^d$ be a wave map with boundary data $\phi_+ = \phi(\cdot, 0): [0, w] \to \bS^d$
and $\phi_- = \phi(0, \cdot): [0, w]\to \bS^d$ such that
\begin{equation}
\max(\|\phi_+'\|_{L^1}, \|\phi_-'\|_{L^1}) \leq \eta_0.
\end{equation}
Then
\begin{equation}
\begin{aligned}
\big|\phi(w, w) - \cR_{\phi(w, 0) + \phi(0, w)}\phi(0, 0)\big|&\leq C\Big(\|\phi_+'\|_{L^1}\Big\|\phi_-' - \frac{1}{w}\big(\phi_-(w) - \phi_-(0)\big)\Big\|_{L^1} \\
&\qquad+ \|\phi_-'\|_{L^1}\Big\|\phi_+' - \frac{1}{w}\big(\phi_+(w) - \phi_+(0)\big)\Big\|_{L^1} \\
&\qquad + \big(\|\phi_+'\|_{L^1} + \|\phi_-'\|_{L^1}\big)\|\phi_+'\|_{L^1}\|\phi_-'\|_{L^1}\Big).
\end{aligned}
\end{equation}
\end{lemma}
\begin{proof}
By considering $(u, v) \mapsto \phi(u/w, v/w)$ instead of $\phi$,
without loss of generality we can assume $w = 1$.

Let $f := -(\partial_u \phi\cdot \partial_v \phi)\phi$, so that
\begin{equation}
f = \Psi(\phi_+, \phi_-, f), \quad \|f\|_{L^1} \lesssim \|\phi_+'\|_{L^1}\|\phi_-'\|_{L^1}.
\end{equation}
We have
\begin{equation}
\label{eq:phi11}
\begin{aligned}
\phi(1, 1) = \phi_+(1) + \phi_-(1) - \phi_+(0) + \int_0^1\int_0^1 f(u, v)\ud v\ud u.
\end{aligned}
\end{equation}
Lemma~\ref{lem:trilinear} yields
\begin{equation}
\label{eq:f-Psiphi}
\begin{aligned}
\|f - \Psi(\phi_+, \phi_-, 0)\|_{L^1} &= \|\Psi(\phi_+, \phi_-, f) - \Psi(\phi_+, \phi_-, 0)\|_{L^1} \\
&\lesssim (\|\phi_+'\|_{L^1} + \|\phi_-'\|_{L^1} + \|f\|_{\ell^1})\|f\|_{L^1} \\
&\lesssim (\|\phi_+'\|_{L^1} + \|\phi_-'\|_{L^1})\|\phi_+'\|_{L^1}\|\phi_-'\|_{L^1}.
\end{aligned}
\end{equation}

Let $\wt\phi_+(u) := (1-u)\phi_+(0) + u\phi_+(1)$
be the affine function which agrees with $\phi_+$ at $u \in \{0, 1\}$,
and similarly $\wt \phi_-(v) := (1-v)\phi_-(0) + v\phi_-(1)$.
Again by Lemma~\ref{lem:trilinear}, we have
\begin{equation}
\begin{aligned}
&\|\Psi(\phi_+, \phi_-, 0) - \Psi(\wt \phi_+, \wt \phi_-, 0)\|_{L^1} \lesssim \\
&\quad \|\phi_+'\|_{L^1}\|\phi_-' - (\phi_-(1) - \phi_-(0))\|_{L^1} \\
&\quad + |\phi_-(1) - \phi_-(0)|\|\phi_+' - (\phi_+(1) - \phi_+(0))\|_{L^1} \\
&\quad + |\phi_+(1) - \phi_+(0)| \times |\phi_-(1) - \phi_-(0)|\times \\
&\qquad\times(\|\phi_+' - (\phi_+(1) - \phi_+(0))\|_{L^1}+\|\phi_-' - (\phi_-(1) - \phi_-(0))\|_{L^1}).
\end{aligned}
\end{equation}
If $\|\phi_+'\|_{L^1} \leq 1$ and $\|\phi_-'\|_{L^1} \leq 1$, then the term given by the 4th and 5th line above
can be absorbed into the 2nd and 3rd line, and we get
\begin{equation}
\label{eq:Psiphi-Psiwtphi}
\begin{aligned}
\|\Psi(\phi_+, \phi_-, 0) - \Psi(\wt \phi_+, \wt \phi_-, 0)\|_{L^1} &\lesssim
\|\phi_+'\|_{L^1}\|\phi_-' - (\phi_-(1) - \phi_-(0))\|_{L^1} \\
&+ \|\phi_-'\|_{L^1}\|\phi_+' - (\phi_+(1) - \phi_+(0))\|_{L^1}.
\end{aligned}
\end{equation}

The function $\Psi(\wt\phi_+, \wt\phi_-, 0)$ is given by \eqref{eq:Psi-at-0}.
Since $\wt \phi_+$ and $\wt \phi_-$ are affine functions, we obtain
\begin{equation}
\label{eq:Psi-phi-tilde}
\begin{aligned}
&\int_0^1 \int_0^1 \Psi(\wt \phi_+, \wt\phi_-, 0)\ud v\ud u = \\
&\qquad = -\big((\phi_+(1) - \phi_+(0))\cdot(\phi_-(1) - \phi_-(0)\big)\big((\phi_+(1) + \phi_-(1))/2\big).
\end{aligned}
\end{equation}
Note that $\phi_-(0) = \phi_+(0) \in \bS^d$ and $|(\phi_+(1) + \phi_-(1))/{2} - \phi_+(0)| \leq \frac 12(|\phi_+(1) - \phi_+(0)| + |\phi_-(1) - \phi_-(0)|)$, so elementary geometry gives
\begin{equation}
\bigg|\cI\bigg(\frac{\phi_+(1) + \phi_-(1)}{2}\bigg) - \frac{\phi_+(1) + \phi_-(1)}{2}\bigg| \lesssim |\phi_+(1) - \phi_+(0)| + |\phi_-(1) - \phi_-(0)|.
\end{equation}
Hence, \eqref{eq:Psi-phi-tilde} and Lemma \ref{lem:R-formula} yield
\begin{equation}
\begin{aligned}
&\cR_{\phi(1, 0) + \phi(0, 1)}\phi(0, 0) - \Big(\phi_+(1) + \phi_-(1) - \phi_+(0) + \int_0^1\int_0^1 \Psi(\wt \phi_+, \wt\phi_-, 0)\ud v\ud u\Big) \lesssim \\
&\qquad\lesssim \big(|\phi_+(1) - \phi_+(0)| + |\phi_-(1) - \phi_-(0)|\big)\times|\phi_+(1) - \phi_+(0)|\times |\phi_-(1) - \phi_-(0)| \\
&\qquad \lesssim\big(\|\phi_+'\|_{L^1} + \|\phi_-'\|_{L^1}\big)\|\phi_+'\|_{L^1}\|\phi_-'\|_{L^1},
\end{aligned}
\end{equation}
and \eqref{eq:phi11} implies
\begin{equation}
\begin{aligned}
|\phi(1, 1) - \cR_{\phi(1, 0) + \phi(0, 1)}\phi(0, 0)| &\lesssim \|f - \Psi(\wt \phi_+, \wt\phi_-, 0)\|_{L^1} \\
&+ \big(\|\phi_+'\|_{L^1} + \|\phi_-'\|_{L^1}\big)\|\phi_+'\|_{L^1}\|\phi_-'\|_{L^1}.
\end{aligned}
\end{equation}
From \eqref{eq:f-Psiphi} and \eqref{eq:Psiphi-Psiwtphi}, we obtain the required bound on the first term of the right hand side
above.
\end{proof}

\begin{proof}[Proof of Theorem~\ref{thm:scheme-conv}]
Let $L, N \in \bN$. For $0 \leq m, n \leq 2^{L+N}$,
let $Z_+(m) := \phi_+\big(\frac{m}{2^{N}}\big)$, $Z_-(n) := \phi_-\big(\frac{n}{2^{N}}\big)$, $Z(m, n) := \phi\big(\frac{m}{2^{N}}, \frac{n}{2^{N}}\big)$.

By Definition~\ref{def:sol-Fe-discr}, $Z: \llbracket 0, 2^{L+N}\rrbracket^2 \to \bS^d$
is a discrete wave map with external forcing $F_\tx e$ given by
\begin{equation}
F_\tx e(m, n) := Z(m+1, n+1) - \cR_{Z(m+1, n) + Z(m, n+1)}Z(m, n).
\end{equation}
Our aim is to prove \eqref{eq:Fe-conv-0} below.

Let $\eta_0$ be the number given by Lemma~\ref{lem:cell-sol}.
Since $\phi_+' \in L^1([0, L])$, there exists $N_0$ such that
\begin{equation}
\max_{m\in\llbracket 0, 2^{L+N}-1\rrbracket}\int_{\frac{m}{2^{N}}}^{\frac{m+1}{2^{N}}}|\phi_+'(u)|\ud u \leq \eta_0, \qquad\text{for all }N \geq N_0.
\end{equation}
By \eqref{eq:pointwise-1}, we thus obtain
\begin{equation}
\max_{n\in\llbracket 0, 2^{L+N}-1\rrbracket}\max_{m\in\llbracket 0, 2^{L+N}-1\rrbracket}\int_{\frac{m}{2^{N}}}^{\frac{m+1}{2^{N}}}\Big|\partial_u \phi\Big(u, \frac{n}{2^{N}}\Big)\Big|\ud u \leq\eta_0, \qquad\text{for all }N \geq N_0.
\end{equation}
Similarly,
\begin{equation}
\max_{m\in\llbracket 0, 2^{L+N}-1\rrbracket}\max_{n\in\llbracket 0, 2^{L+N}-1\rrbracket}\int_{\frac{n}{2^{N}}}^{\frac{n+1}{2^{N}}}\Big|\partial_v \phi\Big(\frac{m}{2^{N}}, v\Big)\Big|\ud v \leq\eta_0, \qquad\text{for all }N \geq N_0.
\end{equation}
By Lemma~\ref{lem:cell-sol},
for all $N \geq N_0$ and $(m, n) \in \llbracket 0, 2^{L+N}-1\rrbracket$ we have
\begin{equation}
\label{eq:Fe-est}
\begin{aligned}
|F_\tx e(m, n)| &\lesssim \int_{\frac{m}{2^{N}}}^{\frac{m+1}{2^{N}}}\Big|\partial_u \phi\Big(u, \frac{n}{2^N}\Big)\Big|\ud u\ \times \\
&\quad \times\int_{\frac{n}{2^{N}}}^{\frac{n+1}{2^{N}}}\Big|\partial_v \phi\Big(\frac{m}{2^{N}}, v\Big) - 2^N\Big(\phi\Big(\frac{m}{2^{N}}, \frac{n+1}{2^{N}}\Big) - \phi\Big(\frac{m}{2^{N}}, \frac{n}{2^{N}}\Big)\Big)\Big|\ud v \\
&+ \int_{\frac{n}{2^{N}}}^{\frac{n+1}{2^{N}}}\Big|\partial_v \phi\Big(\frac{m}{2^N}, v\Big)\Big|\ud v\ \times \\
&\quad \times\int_{\frac{m}{2^{N}}}^{\frac{m+1}{2^{N}}}\Big|\partial_u \phi\Big(u, \frac{n}{2^{N}}\Big) - 2^N\Big(\phi\Big(\frac{m+1}{2^{N}}, \frac{n}{2^{N}}\Big) - \phi\Big(\frac{m}{2^{N}}, \frac{n}{2^{N}}\Big)\Big)\Big|\ud u \\
&+ \bigg(\int_{\frac{m}{2^{N}}}^{\frac{m+1}{2^{N}}}\Big|\partial_u \phi\Big(u, \frac{n}{2^N}\Big)\Big|\ud u + \int_{\frac{n}{2^{N}}}^{\frac{n+1}{2^{N}}}\Big|\partial_v \phi\Big(\frac{m}{2^N}, v\Big)\Big|\ud v\bigg) \times \\
&\quad \times \int_{\frac{m}{2^{N}}}^{\frac{m+1}{2^{N}}}\Big|\partial_u \phi\Big(u, \frac{n}{2^N}\Big)\Big|\ud u \times \int_{\frac{n}{2^{N}}}^{\frac{n+1}{2^{N}}}\Big|\partial_v \phi\Big(\frac{m}{2^N}, v\Big)\Big|\ud v.
\end{aligned}
\end{equation}

By Lemma~\ref{lem:L1-lemma}, for all $u \in [0, 2^{L}]$ we have
\begin{equation}
\begin{aligned}
&\lim_{N\to \infty}\sum_{n=0}^{2^{L+N}-1}\int_\frac{n}{2^N}^\frac{n+1}{2^N}\Big|\partial_v \phi(u, v) - 2^N\Big(\phi\Big(u, \frac{n+1}{2^{N}}\Big) - \phi\Big(u, \frac{n}{2^{N}}\Big)\Big)\Big|\ud v \\
&\quad\qquad = \lim_{N\to \infty}\big\|T_{L+N}\big(\partial_v \phi(u, \cdot)\big)\big\|_{L^1} = 0,
\end{aligned}
\end{equation}
where $T_{L+N}$ is defined by \eqref{eq:TN-def}.

By Definition~\ref{def:strong}, the mapping $u \mapsto \partial_v\phi(u, \cdot)$ is continous
from $[0, 2^L]$ to $L^1([0, 2^L])$, hence uniformly continuous,
so $\|T_{L+N}\|_{\scrL(L^1)}\leq 2$ implies
\begin{equation}
\lim_{N\to\infty}\sup_{0 \leq u \leq 2^L}\sum_{n=0}^{2^{L+N}-1}\int_\frac{n}{2^N}^\frac{n+1}{2^N}\Big|\partial_v \phi(u, v) - 2^N\Big(\phi\Big(u, \frac{n+1}{2^{N}}\Big) - \phi\Big(u, \frac{n}{2^{N}}\Big)\Big)\Big|\ud v = 0.
\end{equation}
In particular
\begin{equation}
\begin{aligned}
\lim_{N\to\infty}\max_{m \in \llbracket 0, 2^{L+N}-1\rrbracket}&\sum_{n=0}^{2^{L+N}-1}\int_{\frac{n}{2^{N}}}^{\frac{n+1}{2^{N}}}\Big|\partial_v \phi\Big(\frac{m}{2^{N}}, v\Big) \\
 &- 2^N\Big(\phi\Big(\frac{m}{2^{N}}, \frac{n+1}{2^{N}}\Big) - \phi\Big(\frac{m}{2^{N}}, \frac{n}{2^{N}}\Big)\Big)\Big|\ud v = 0.
\end{aligned}
\end{equation}
Using \eqref{eq:pointwise-1}, for all $N$ we also have
\begin{equation}
\begin{aligned}
\sum_{m=0}^{2^{L+N}-1}\max_{n \in \llbracket 0, 2^{L+N}-1\rrbracket} \int_{\frac{m}{2^{N}}}^{\frac{m+1}{2^{N}}}\Big|\partial_u \phi\Big(u, \frac{n}{2^N}\Big)\Big|\ud u = \sum_{m=0}^{2^{L+N}-1} \int_{\frac{m}{2^{N}}}^{\frac{m+1}{2^{N}}}\big|\phi_+'(u)\big|\ud u = \|\phi_+'\|_{L^1}.
\end{aligned}
\end{equation}
By summing first with respect to $n$, and then with respect to $m$,
we thus obtain that the sum in $m$ and $n$ of the term given by the
1st and 2nd line in \eqref{eq:Fe-est} tends to $0$ as $N \to \infty$.
The term given by the 3rd and 4th line is treated similarly.

Consider now the 5th and 6th line. For all $N$ we have
\begin{equation}
\sum_{m, n}\bigg(\int_{\frac{m}{2^{N}}}^{\frac{m+1}{2^{N}}}\Big|\partial_u \phi\Big(u, \frac{n}{2^N}\Big)\Big|\ud u \times \int_{\frac{n}{2^{N}}}^{\frac{n+1}{2^{N}}}\Big|\partial_v \phi\Big(\frac{m}{2^N}, v\Big)\Big|\ud v\bigg) = \|\phi_+'\|_{L^1}\|\phi_-'\|_{L^1}.
\end{equation}
Using \eqref{eq:pointwise-1} and \eqref{eq:pointwise-2}, we also have
\begin{equation}
\lim_{N\to \infty}\max_{m, n}\int_{\frac{m}{2^{N}}}^{\frac{m+1}{2^{N}}}\Big|\partial_u \phi\Big(u, \frac{n}{2^N}\Big)\Big|\ud u = \lim_{N\to \infty}\max_m \int_{\frac{m}{2^{N}}}^{\frac{m+1}{2^{N}}}\big|\phi_+'(u)\big|\ud u = 0
\end{equation}
and
\begin{equation}
\lim_{N\to \infty}\max_{m, n}\int_{\frac{n}{2^{N}}}^{\frac{n+1}{2^{N}}}\Big|\partial_v \phi\Big(\frac{m}{2^N}, v\Big)\Big|\ud v = \lim_{N\to \infty}\max_n \int_{\frac{n}{2^{N}}}^{\frac{n+1}{2^{N}}}\big|\phi_-'(v )\big|\ud v = 0.
\end{equation}
We thus conclude that
\begin{equation}
\label{eq:Fe-conv-0}
\lim_{N\to\infty}\|F_\tx e\|_{\ell^1} = 0.
\end{equation}
By Lemma~\ref{lem:l-time-pertb}, $\lim_{N\to\infty}\|Z - Y\|_{\ell^\infty} = 0$,
and the uniform continuity of $\phi$ on bounded sets yields \eqref{eq:scheme-conv}.
\end{proof}

\section{Wave maps with random boundary data}
\label{sec:random}
\subsection{Notion of solution for random boundary data}
In Section~\ref{sec:cauchy}, we presented well-posedness results
in the case of absolutely continuous boundary data.
Our goal in the present section is to consider a continuous random field as boundary data
and study the corresponding random field of solutions of the characteristic Cauchy problem.
\begin{definition}
\label{def:solution}
Let $\phi_-, \phi_+ : \Omega \to X([0, \infty))$ be such that $\phi_-^{\omega}(0) = \phi_+^{\omega}(0)$
with probability one.
We say that a random field $\phi: \Omega' \to X([0, \infty)^2)$ is a solution
of \eqref{eq:wm-cauchy} if there exists an increasing sequence $N_k$ such that
\begin{equation}
\label{eq:conv-psiNk}
\Phi_{N_k}(\phi_+, \phi_-) \to \phi \qquad\text{in distribution on $X([0, \infty)^2)$.}
\end{equation}
\end{definition}

Note that solutions are defined here only as \emph{probability laws}
on $X([0, \infty)^2)$, which is highlighted above by the fact that
$\phi$ is defined on the sample space $\Omega'$, distinct from $\Omega$.
This is in contrast with the ``individual trajectory'' approach,
which would aim at defining the solution as a random field $\wt \phi:\Omega \to X([0, \infty)^2)$
such that, for almost all $\omega\in \Omega$, $\wt \phi^{(\omega)}$
be a solution of \eqref{eq:wm-cauchy} with the boundary data $(\phi_+^{(\omega)}, \phi_-^{(\omega)})$ in some \emph{deterministic} sense which would have to be specified.

Since the individual trajectory approach is more natural from the PDE viewpoint,
it is relevant to compare it with the seemingly more artificial Definition~\ref{def:solution}.
As a consequence of Theorem~\ref{thm:scheme-conv},
we find that for absolutely continuous data the result is the same.
Recall that $\Phi$ is the solution map obtained in Theorem~\ref{thm:keel-tao}.
\begin{corollary}
\label{cor:the-same}
\label{def:individual}
Let $\phi_-, \phi_+ : \Omega \to X([0, \infty))$ be such that $\phi_-^{\omega}(0) = \phi_+^{\omega}(0)$
with probability one and $\phi_+', \phi_-' \in L^1_\tx{loc}([0, \infty))$ with probability one.
Then $\omega \mapsto \wt \phi^{(\omega)} := \Phi(\phi_+^{(\omega)}, \phi_-^{(\omega)})$ is a solution
of \eqref{eq:wm-cauchy} in the sense of Definition~\ref{def:solution},
and any other solution $\phi: \Omega' \to X([0, \infty)^2)$
has the same distribution as $\wt \phi$.
\end{corollary}
\begin{proof}
Let $\phi$ be a solution and $N_k$ the corresponding increasing sequence such that \eqref{eq:conv-psiNk} holds.
By Theorem~\ref{thm:scheme-conv}, we have
$\Phi_{N_k}(\phi_+^{(\omega)}, \phi_-^{(\omega)}) \to \wt \phi^{(\omega)}$ in $X([0, \infty)^2)$ with probability 1,
in particular in distribution. Uniqueness of the limit implies that $\phi$ and $\wt \phi$ have the same distribution.
\end{proof}
The advantage of Definition~\ref{def:solution} is that it is meaningful even if
the boundary data are not absolutely continuous.
In the next section we focus on the case where the boundary data are given by the standard Brownian motion.
\subsection{Discrete wave maps with Markov chain as boundary data}
\label{ssec:heat-brownian}
The standard heat kernel on $\bS^d$ is denoted
\begin{equation}
\label{eq:heat-kernel}
p: (0, \infty)\times \bS^{d} \times \bS^d \to (0, \infty).
\end{equation}
We adopt the convention that $f(t, y) = p(t, x_0, y)$ solves $\partial_t f = \frac 12 \Delta_y f$, where $\Delta$ is the Laplace-Beltrami operator on the sphere.
From the uniqueness of the heat kernel, we deduce that for any linear isometry $U: \bR^{d+1} \to \bR^{d+1}$ we have
\begin{equation}
\label{eq:heat-symmetry}
p(t, Ux, Uy) = p(t, x, y),\qquad\text{for all }(t, x, y) \in (0, \infty)\times \bS^d.
\end{equation}
We will use Gaussian bounds on the heat kernel. Such estimates were first obtained by Varadhan \cite{Varadhan-1967},
but for our purposes the statement of Cheng, Li and Yau~\cite{Cheng-Li-Yau} is more convenient.
\begin{lemma}
\label{lem:Cheng-Li-Yau}
There exists $C$ such that for all $A \geq \sqrt{6d}$, $x \in \bS^d$ and $t \in \big(0, \frac 18\big]$
\begin{equation}
\int_{|x - y| \geq A\sqrt{{-}t\log t}}p(t, x, y)\ud y \leq Ct^\frac{A^2}{8}.
\end{equation}
\end{lemma}
\begin{proof}
By symmetry, it suffices to check the inequality for one fixed $x \in \bS^d$.
By \cite[Theorem 3]{Cheng-Li-Yau}, for all $\beta > 1$ and $T > 0$
there exists $C$ such that for all $R \geq 0$ and $t \in (0, T]$
\begin{equation}
\int_{\bS^d \setminus B(x, R)}p(t, x, y)^2\ud y \leq Ct^{-\frac d2}\exp\Big({-}\frac{R^2}{2\beta t}\Big),
\end{equation}
where $B(x, R)$ is the geodesic ball of center $x$ and radius $R$.
Setting $\beta := \frac 32$, $T := \frac 18$ and $R := A\sqrt{{-}t\log t}$,
and using the fact that the Euclidean distance in $\bR^{d+1}$ is smaller than the geodesic distance on $\bS^d$, we obtain
\begin{equation}
\int_{|x-y|\geq A\sqrt{{-}t\log t}}p(t, x, y)^2\ud y \leq Ct^{\frac{A^2}{3}-\frac d2}
\leq Ct^{\frac{A^2}{4}}, \qquad\text{for all }A \geq \sqrt{6d}, t \in \Big(0, \frac 18\Big].
\end{equation}
An application of the Cauchy-Schwarz inequality finishes the proof.
\end{proof}

\begin{definition}
Let $N \in \bN$ and $t > 0$. We say that a random sequence $X: \Omega \times \llbracket 0, N\rrbracket \to \bS^d$ is a \emph{heat Markov chain} with parameter $t$ if its probability density is given by
\begin{equation}
|\bS^d|^{-1}\prod_{n=0}^{N-1}p\big(t, X(n), X(n+1)\big).
\end{equation}
\end{definition}
\begin{remark}
In particular, for all $n$ the probability density of $X(n)$ is $|\bS^d|^{-1}$,
thus we assume the Markov chain starts from a point on the sphere chosen randomly with the uniform distribution.
\end{remark}
\begin{lemma}
\label{lem:reflection}
For all $t > 0$ and $P, Q, S \in \bS^d$ there is the equality
\begin{equation}
\label{eq:reflection}
p(t, P, \cR_{P+Q}S)p(t, \cR_{P+Q}S, Q) = p(t, P, S)p(t, S, Q).
\end{equation}
\end{lemma}
\begin{proof}
Consider the isometry $U := \cR_{P+Q}$. By \eqref{eq:R-exchange}, we have $U P = Q$ and $UQ = P$,
so \eqref{eq:heat-symmetry} yields $p(t, P, US) = p(t, Q, S) = p(t, S, Q)$
and $p(t, US, Q) = p(t, S, P) = p(t, P, S)$.
\end{proof}

The key observation is the following lemma, expressing the fact
that Markov chains with transition probabilities given
by the keat kernel are preserved by the discrete wave maps equation.
\begin{lemma}
\label{lem:reflections}
Let $M, N \in \bN$, $t > 0$, and let $Y_+: \Omega \times \llbracket 0, M \rrbracket \to \bS^d$,
$Y_-: \Omega \times \llbracket 0, N\rrbracket \to \bS^d$ be random sequences such that
$Y_+(0) = Y_-(0)$ with probability 1, and
\begin{equation}
\begin{aligned}
&\big(Y_-(N), Y_-(N-1), \ldots, Y_-(1), Y_-(0) = Y_+(0), \\
&\qquad\qquad Y_+(1), \ldots, Y_+(M-1), Y_+(M)\big) \in (\bS^d)^{M+N+1}
\end{aligned}
\end{equation}
is a heat Markov chain with parameter $t$.
Let $Y: \Omega \times \llbracket 0, M \rrbracket \times \llbracket 0, N \rrbracket \to \bS^d$
be the solution of the discrete wave maps equation with boundary data $Y_\pm$.

If $(m_j, n_j)_{j=0}^{M+N}$ is a sequence such that $(m_0, n_0) = (0, N)$, $(m_{M+N}, n_{M+N}) = (M, 0)$,
and for every $j$ either $(m_{j+1}, n_{j+1}) = (m_j + 1, n_j)$, or $(m_{j+1}, n_{j+1}) = (m_j, n_j - 1)$,
then the random sequence $(Y(m_j, n_j))_{j=0}^{M+N}$ is a heat Markov chain with parameter $t$.
\end{lemma}
\begin{remark}
Graphically, $(m_j, n_j)_{j=0}^{M+N}$ is a sequence of coordinates starting at the upper left corner,
ending at the lower right corner, and going down or right at each step.
Such a path is (space or light)-like, so the independence of the jumps is expected
because of the finite speed of propagation.
\end{remark}
\begin{proof}[Proof of Lemma~\ref{lem:reflections}]
We argue by induction with respect to $\sum_{j=0}^{M+N}(m_j + n_j)$.
This sum is minimal for the sequence of indices going down to the bottom and then right all the way to the corner, that is
\begin{equation}
(m_j, n_j)_{j=0}^{M+N} = \big((0, N), (0, N-1), \ldots, (0, 1), (0, 0), (1, 0), \ldots, (M-1, 0), (M, 0)\big),
\end{equation}
which corresponds to the boundary data and thus is indeed a heat Markov chain, by assumption.

Let $(m_j, n_j)_{j=0}^{M+N}$ be some other sequence. Then there exists $k \in \{1, \ldots, M+N-1\}$ such that
\begin{equation}
(m_{k}, n_{k}) = (m_{k - 1}+1, n_{k-1}) = (m_{k + 1}, n_{k+1}+1),
\end{equation}
in other words a place in the sequence where it moves right and then down
(if such $k$ did not exist, then the sequence would have to correspond to $N$ moves down
followed by $M$ moves to the right, which is the sequence corresponding to the boundary data).

We let $(m_j', n_j')_{j=0}^{M+N}$ be the sequence defined by $(m_j', n_j') = (m_j, n_j)$ if $j \neq k$
and $(m_{k}', n_{k}') = (m_{k} - 1, n_{k} - 1) = (m_{k-1}, n_{k+1})$.
We thus have
\begin{equation}
Y(m_j', n_j') = Y(m_j, n_j), \qquad\text{for all }j \neq k
\end{equation}
and
\begin{equation}
\begin{aligned}
Y(m_k, n_k) &= \cR_{Y(m_k-1, n_k) + Y(m_k, n_k-1)}Y(m_k-1, n_k-1) \\
&= \cR_{Y(m_{k-1}, n_{k-1}) + Y(m_{k+1}, n_{k+1})}Y(m_k-1, n_k-1) \\
&= \cR_{Y(m_{k-1}', n_{k-1}') + Y(m_{k+1}', n_{k+1}')}Y(m_k', n_k').
\end{aligned}
\end{equation}

By the induction hypothesis, the random sequence $\big(Y(m_j', n_j')\big)_{j=0}^{M+N}$
is a heat Markov chain.
By Lemma~\ref{lem:reflection}, the same is true for the sequence $\big(Y(m_j, n_j)\big)_{j=0}^{M+N}$.
\end{proof}
\begin{corollary}
\label{cor:discr-inv}
Let $M, N, Y_+, Y_-$ and $Y$ be as in Lemma~\ref{lem:reflections}.
Let $(m_0, n_0) \in \llbracket 0, M \rrbracket \times \llbracket 0, N\rrbracket$
and $(m_j, n_j) \in \llbracket 0, M - m_0 \rrbracket \times \llbracket 0, N - n_0\rrbracket$
for $j \in \{1, 2, \ldots, J\}$.
Then the random vectors $\big(Y(m_j, n_j)\big)_{j=1}^J$ and $\big(Y(m_0 + m_j, n_0 + n_j)\big)_{j=1}^J$ have the same distribution.
\end{corollary}
\begin{proof}
Let $\wt m := \max\{m_j: 1 \leq j \leq J\}$ and $\wt n := \max\{n_j: 1 \leq j \leq J\}$.
By Lemma~\ref{lem:discr-exist}, there exists a Borel map $\cY:(\bS^d)^{m_0+1}\times (\bS^d)^{n_0+1} \to (\bS^d)^J$ such that
\begin{equation}
\big(Y(m_j, n_j)\big)_{j=1}^J = \cY\big(Y_+\vert_{\llbracket 0, \wt m\rrbracket}, Y_-\vert_{\llbracket 0, \wt n\rrbracket}\big)
\end{equation}
and
\begin{equation}
\big(Y(m_0 + m_j, n_0 + n_j)\big)_{j=1}^J = \cY\big(Y(m_0 + \cdot, n_0)\vert_{\llbracket 0, \wt m\rrbracket}),
Y(m_0, n_0 + \cdot)\vert_{\llbracket 0, \wt n\rrbracket}\big).
\end{equation}
From Lemma~\ref{lem:reflections}, we infer that $\big(Y_+\vert_{\llbracket 0, \wt m\rrbracket}, Y_-\vert_{\llbracket 0, \wt n\rrbracket}\big)$ and $\big(Y(m_0 + \cdot, n_0)\vert_{\llbracket 0, \wt m\rrbracket}),
Y(m_0, n_0 + \cdot)\vert_{\llbracket 0, \wt n\rrbracket}\big)$ have the same distribution,
hence $\big(Y(m_j, n_j)\big)_{j=1}^J$ and $\big(Y(m_0 + m_j, n_0 + n_j)\big)_{j=1}^J$ have the same distribution as well.
\end{proof}

We are ready to estimate the modulus of continuity for discrete wave maps with heat Markov chains
as boundary data.
\begin{lemma}
\label{lem:main-lemma}
There exists $C = C(d)$ such that for all $L \in \bN$ and $N \in \{4, 5, \ldots\}$ the following is true.
Let $Y_+, Y_-: \Omega \times \llbracket 0, 2^{N+L}\rrbracket \to \bS^d$
be such that $Y_+(0) = Y_-(0)$ with probability 1 and the random sequence
\begin{equation}
\big(Y_-(2^{N+L}), \ldots, Y_-(0) = Y_+(0), \ldots, Y_+(2^{N+L})\big)
\end{equation}
is a heat Markov chain with parameter $2^{-N}$,
and let $Y: \Omega \times \llbracket 0, 2^{N+L}\rrbracket^2 \to \bS^d$
be the discrete wave map corresponding to the boundary data $Y_\pm$. Then for all $A \geq 0$
\begin{equation}
\label{eq:cont-mod-ineq}
\begin{aligned}
&\bP\bigg(\bigg\{\max_{\substack{0 \leq m, n, \sh m, \sh n \leq 2^{N+L} \\ (m, n) \neq (\sh m, \sh n)}}\frac{|Y(\sh m, \sh n) - Y(m, n)|}{h(2^{-N}(|\sh m-m| + |\sh n-n|))} \geq 32A\bigg\} \cup \\
&\qquad \cup \big\{\|\delta_m Y\|_{\ell^\infty} \geq 8Ah(2^{-N})\big\} \cup \big\{\|\delta_n Y\|_{\ell^\infty}
\geq 8Ah(2^{-N})\big\}\bigg)\leq C4^{L-A^2 / 8},
\end{aligned}
\end{equation}
where $h(\rho) := \sqrt{-\rho \log \rho}$ for all $\rho \in (0, \eee^{-1}]$, $h(\rho) := \eee^{-\frac 12}$ for all $\rho \geq \eee^{-1}$.
\end{lemma}
\begin{proof}
We can assume $A > \max(\sqrt{6d}, \sqrt {32})$ (for smaller $A$ the bound is evident
if $C$ is large enough).
If $|\sh m - m| \geq 2^{N-3}$ or $|\sh n - n| \geq 2^{N-3}$, then
\begin{equation}
\frac{|Y(\sh m, \sh n) - Y(m, n)|}{h(2^{-N}(|\sh m-m| + |\sh n-n|))} \leq \frac{2}{h(2^{-N}2^{N-3})} = \frac{2}{\sqrt{\frac 18\log 8}} \leq 8 < 32A,
\end{equation}
so such $m, n, \sh m, \sh n$ can be ignored for the purpose of proving \eqref{eq:cont-mod-ineq}.

For $0 \leq M \leq N-3$, set
\begin{equation}
\begin{aligned}
a_M &:= \bP\bigg(\sup_{\substack{0 \leq j < 2^{N-M+L} \\ 0 \leq k \leq 2^{N-M+L}}}
\frac{|Y((j+1)2^M, k2^M) - Y(j2^M, k2^M)|}{h(2^{M-N})}\geq A\bigg), \\
b_M &:= \bP\bigg(\sup_{\substack{0 \leq j \leq 2^{N-M+L} \\ 0 \leq k < 2^{N-M+L}}}
\frac{|Y(j2^M, (k+1)2^M) - Y(j2^M, k2^M)|}{h(2^{M-N})}\geq A\bigg).
\end{aligned}
\end{equation}
We claim that
\begin{equation}
\label{eq:cont-mod-step1}
\max(\max_M a_M, \max_M b_M) \leq C2^{2(N-M+L)+1}2^{(M-N)A^2/8}.
\end{equation}
We do the computation for $a_M$, the one for $b_M$ being analogous.

Fix $0 \leq j < 2^{N-M+L}$ and $0 \leq k \leq 2^{N-M+L}$.
By Lemma~\ref{lem:reflections}, $(x, y) = (Y(j2^M, k2^M), Y((j+1)2^M, k2^M))$ has density $|\bS^d|^{-1}p(2^{M-N}, x, y)$.
Letting $t = 2^{M-N}$, we get by Lemma~\ref{lem:Cheng-Li-Yau}
\begin{equation}
\begin{aligned}
\bP\bigg(\frac{|Y((j+1)2^M, k2^M) - Y(j2^M, k2^M)|}{h(2^{M-N})}\geq A\bigg) &=
\int_{|y - x| \geq Ah(t)}p(t, x, y)\ud y \\
&\leq Ct^\frac{A^2}{8} = C2^{(M-N)A^2/8}.
\end{aligned}
\end{equation}
Estimating the probability of the union by the sum of probabilities for all possible $j$ and $k$, we obtain \eqref{eq:cont-mod-step1}.

Let $X_A$ be the event that there exist $M \in \{0, 1, \ldots, N-3\}$
and a pair $(j, k)$ such that
\begin{equation}
\frac{|Y((j+1)2^M, k2^M) - Y(j2^M, k2^M)|}{h(2^{M-N})}\geq {A}\quad\text{or}\quad
\frac{|Y(j2^M, (k+1)2^M) - Y(j2^M, k2^M)|}{h(2^{M-N})}\geq {A}.
\end{equation}
Taking the sum in $M$ in \eqref{eq:cont-mod-step1}, we obtain
\begin{equation}
\begin{aligned}
\bP(X_A) &\leq 2C\sum_{M=0}^{N-3} 2^{2(N-M+L)+1}2^{(M-N)A^2/8}
 = 4^{L+1}C\sum_{M=0}^{N-3}2^{(M-N)(A^2/8-2)} \\
&\leq 4^{L+1}C \sum_{l \geq 3}2^{-l(A^2/8-2)}
\leq 4^{L+1}C 4^{-A^2/8+2}\sum_{l\geq 1}{2^{-2l}},
\end{aligned}
\end{equation}
where the last inequality follows from $A^2/8 \geq 4$.

Taking $M = 0$ in the definition of the event $X_A$, we see that the fact that $X_A$ does not occur implies
\begin{equation}
\|\delta_m Y\|_{\ell^\infty} \leq Ah(2^{-N}) < 8Ah(2^{-N}), \qquad \|\delta_n Y\|_{\ell^\infty} < 8Ah(2^{-N}).
\end{equation}

It remains to show that if $X_A$ does not occur, then
\begin{equation}
\label{eq:mod-cont-bound}
\frac{|Y(\sh m, \sh n) - Y(m, n)|}{h(2^{-N}(|\sh m-m| + |\sh n-n|))} < 32A
\end{equation}
for all $0 \leq m, n, \sh m, \sh n \leq 2^{N+L}$ such that $0 < \max(|\sh m-m|, |\sh n-n|) < 2^{N-3}$.
Fix any such $m, n, \sh m, \sh n$ and let $M_0$ be the unique integer such that $2^{M_0} \leq \max(|\sh m-m|, |\sh n-n|) < 2^{M_0 + 1}$. Note that $0 \leq M_0 < N-3$.
We claim that there exists $j_0 \in \llbracket 0, 2^{N+L-M_0}\rrbracket$ such that
\begin{equation}
\label{eq:seq-j-base}
\max(|m - j_02^{M_0}|, |\sh m - j_02^{M_0}|) < 2^{M_0+1}.
\end{equation}
Without loss of generality, assume $m \geq \sh m$ and let $j_0$ be the integer part of $m2^{-M_0}$.
Then it is clear that $j_02^{M_0} \leq m < j_02^{M_0}+2^{M_0}$, and $j_02^{M_0} + 2^{M_0} > m \geq \sh m > m - 2^{M_0+1} \geq j_02^{M_0} - 2^{M_0+1}$.
Analogously, let $k_0 \in \llbracket 0, 2^{N+L-M_0}\rrbracket$ be such that
\begin{equation}
\max(|n - k_02^{M_0}|, |\sh n - k_02^{M_0}|) < 2^{M_0+1}.
\end{equation}

We estimate $|Y(m, n) - Y(j_02^{M_0}, k_02^{M_0})|$ as follows.
We define recursively a sequence $j_1, \ldots j_{M_0}, j_{M_0+1}$
such that for all $l \geq 1$
\begin{gather}
\label{eq:seq-j-1}
|m - j_l 2^{M_0-l}| < 2^{M_0-l+1}, \\
\label{eq:seq-j-2}
j_l \in \{2j_{l-1} - 2, 2j_{l-1}, 2j_{l-1}+2\}.
\end{gather}
Note that \eqref{eq:seq-j-base} shows that \eqref{eq:seq-j-1} holds for $l = 0$.
Once $j_0, \ldots, j_{l-1}$ are constructed, we see from \eqref{eq:seq-j-1} that there exists $\iota_{l-1} \in \{{-}1, 0, 1\}$
such that $|m - (j_{l-1}+\iota_{l-1})2^{M_0 - l}| < 2^{M_0-l}$, and it suffices to set $j_l := 2(j_{l-1}+\iota_{l-1})$.

Analogously, we define $k_1, \ldots, k_{M_0}, k_{M_0+1}$ such that for all $l \geq 1$
\begin{gather}
\label{eq:seq-k-1}
|n - k_l 2^{M_0-l}| < 2^{M_0-l+1}, \\
\label{eq:seq-k-2}
k_l \in \{2k_{l-1} - 2, 2k_{l-1}, 2k_{l-1}+2\}.
\end{gather}

Since $X_A$ does not occur, for every $l \in \{0, \ldots, M_0\}$ we have
\begin{equation}
\begin{aligned}
&|Y(j_{l} 2^{M_0-l}, k_{l}2^{M_0-l}) - Y(j_{l+1}2^{M_0-l-1}, k_{l}2^{M_0-l})| \\
&\qquad = |Y(j_{l} 2^{M_0-l}, k_{l}2^{M_0-l}) - Y((j_{l} + \iota_{l}) 2^{M_0-l}, k_{l}2^{M_0-l})| < Ah(2^{M_0-l-N}).
\end{aligned}
\end{equation}
Similarly,
\begin{equation}
\begin{aligned}
|Y(j_{l+1} 2^{M_0-l-1}, k_{l}2^{M_0-l}) - Y(j_{l+1}2^{M_0-l-1}, k_{l+1}2^{M_0-l-1})| < Ah(2^{M_0-l-N}),
\end{aligned}
\end{equation}
hence
\begin{equation}
|Y(j_{l} 2^{M_0-l}, k_{l}2^{M_0-l}) - Y(j_{l+1}2^{M_0-l-1}, k_{l+1}2^{M_0-l-1})| < 2A h(2^{M_0-l-N}),
\end{equation}
Taking the sum in $l$, we obtain
\begin{equation}
\label{eq:dist-to-vertex}
\begin{aligned}
|Y(m, n) - Y(j_02^{M_0}, k_02^{M_0})| &\leq 2A\sum_{l=0}^{M_0}h(2^{M_0-l-N}) \leq 2A \sum_{l=0}^\infty h(2^{-l}\rho)
,\end{aligned}
\end{equation}
where $\rho := 2^{M_0 - N} \in (0, 1/16]$. We claim that
\begin{equation}
\label{eq:rho-sum}
\sum_{l=0}^\infty h(2^{-l}\rho) < 8h(\rho), \qquad\text{for all }0 < \rho \leq 1/16.
\end{equation}
Indeed,
\begin{equation}
\begin{aligned}
\sum_{l=0}^\infty \sqrt{2^{-l}\rho({-}\log \rho + l\log 2)} < \sqrt{-\rho\log \rho}\sum_{l=0}^\infty 2^{-\frac l2}
+ \sqrt\rho\log 2\sum_{l=1}^\infty \sqrt{l2^{-l}}.
\end{aligned}
\end{equation}
We have $\sum_{l=0}^\infty 2^{-\frac l2} = 2 + \sqrt 2$, $\sqrt\rho\log 2 \leq \frac 12 h(\rho)$ and
\begin{equation}
\sum_{l=1}^\infty \sqrt{l2^{-l}} \leq \bigg(\sum_{l=1}^\infty l2^{-\frac l2}\bigg)^\frac 12\bigg(\sum_{l=1}^\infty 2^{-\frac l2}\bigg)^\frac 12 = (1+\sqrt 2)\sqrt{2 + \sqrt 2},
\end{equation}
and elementary arithmetics leads to \eqref{eq:rho-sum}.

From \eqref{eq:dist-to-vertex} and \eqref{eq:rho-sum}, we have $|Y(m, n) - Y(j_02^{M_0}, k_02^{M_0})| < 16Ah(\rho)$.
Analogously, $|Y(\sh m, \sh n) - Y(j_02^{M_0}, k_02^{M_0})| < 16Ah(\rho)$.
Since $2^{-N}(|\sh m - m| + |\sh n - n|) \geq \rho$, we obtain \eqref{eq:mod-cont-bound}.
\end{proof}

Let $x_0 \in \bS^d$ and let $B^{(x_0)}: \Omega \times [0, \infty) \to \bS^d$
be the standard Brownian motion starting from $x_0$.
Then
\begin{equation}
\label{eq:transition}
\bP\big(B^{(x_0)}(t) \in A\big) = \int_{A} p(t, x_0, y)\ud y, \qquad\text{for all }t > 0\text{ and }A \in \cB(\bS^d),
\end{equation}
see \cite[Section 4.1]{Hsu}.

Our boundary data will be a pair of Brownian motions $\psi_+, \psi_-: \Omega \to X([0, \infty))$, starting from the same
point of $\bS^d$, chosen randomly with uniform distribution.
They can be obtained by letting $\psi_0: \Omega \to X(\bR)$ be the standard
Brownian motion starting from a uniformly chosen random point,
and setting $\psi_+(u) := \psi_0(u)$ and $\psi_-(v) := \psi_0(-v)$ for all $u, v \geq 0$.

\subsection{Existence of wave maps with Brownian motion as boundary data}
\label{sec:prochorow}

\begin{theorem}
\label{thm:main}
Let $\phi_+, \phi_-: \Omega \to X([0, \infty))$ be independent standard Brownian motions
starting from the same uniformly chosen point on $\bS^d$.
There exists a solution $\phi: \Omega' \to X([0, \infty)^2)$ of \eqref{eq:wm-cauchy}
in the sense of Definition~\ref{def:solution}. It has the following properties:
\begin{enumerate}[1)]
\item
\label{it:invariance}
for all $(u_0, v_0) \in [0, \infty)^2$, the random field $\xi: \Omega' \to X([0, \infty)^2)$ defined by
$\xi^{(\omega)}(u, v) := \phi^{(\omega)}(u_0 + u, v_0 + v)$
has the same distribution as $\phi$,
\item
\label{it:modulus}
for all $\wt L > 0$,
\begin{equation}
\label{eq:modulus}
\sup_{\substack{0 \leq u, v, \sh u, \sh v \leq \wt L \\ (u, v) \neq (\sh u, \sh v)}}\frac{|\phi(\sh u, \sh v) - \phi(u, v)|}{h(|(\sh u - u, \sh v - v)|)} < \infty,\qquad\text{with probability 1.}
\end{equation}
\end{enumerate}
\end{theorem}
\begin{remark}
The property \ref{it:invariance} expresses the invariance of the field
under time-like translations.
This is not surprising, since the constructed random field
is closely related to the Gibbs measure for the wave maps equation, see Section~\ref{ssec:gibbs} below.
\end{remark}
\begin{remark}
Even though equation \eqref{eq:wm-cauchy} is invariant by Lorentz transformations, the field $\phi$ is not.
This fact is not surprising if we recall the physical motivation provided in Section~\ref{ssec:setting},
which was based on \emph{non-relativistic} thermodynamics.
It could be interesting to consider a relativistic version of our problem,
a task we did not strive for due to our insufficient understanding of the underlying physical theories.
\end{remark}
\begin{remark}
A natural open problem is to prove or disprove the uniqueness
of the probability distribution of a random field solving \eqref{eq:wm-cauchy}
in the sense of Definition~\ref{def:solution}
(or according to some more restrictive notion of solution).
\end{remark}
\begin{remark}
Applying the L\'evy Modulus of Continuity Theorem to the boundary data $(\phi_+, \phi_-)$,
we deduce that the estimate \eqref{eq:modulus} is optimal,
in the sense that it would fail if $h$ was replaced by a function
$\wt h$ such that $\wt h(\rho) \ll h(\rho)$ as $\rho \to 0^+$.
In order to only prove the existence of a solution $\phi$,
it would be sufficient to consider, in Lemma~\ref{lem:main-lemma} above, $h(\rho) := \rho^\alpha$ for any $\alpha < \frac 12$.
\end{remark}

The rest of this section is devoted to a proof of Theorem~\ref{thm:main}.
Our strategy is to show that the sequence of random variables $\Phi_N(\phi_+, \phi_-): \Omega \to C([0, \infty)^2)$ is tight.
We have the following pre-compactness criterion in $C([0, \infty)^2)$.
\begin{lemma}
\label{lem:pre-compact}
Let $\rho: [0, \infty) \to [0, \infty)$ be a continuous function such that $\rho(0) = 0$.
For any sequence $(A_L)_L \geq 0$ the set of $\psi \in C([0, \infty)^2)$ such that
\begin{equation}
\label{eq:pre-compact-1}
\sup_{0 \leq u, v < \infty}|\psi(u, v)| \leq 1
\end{equation}
and
\begin{equation}
\label{eq:pre-compact-2}
\sup_{\substack{0\leq u, v, \sh u, \sh v \leq 2^L \\ (u, v) \neq (\sh u, \sh v)}}\frac{|\psi(\sh u, \sh v) - \psi(u, v)|}{\rho(|\sh u - u| + |\sh v - v|)} \leq A_L, \qquad\text{for all }L \in \bN,
\end{equation}
is compact in $C([0, \infty)^2)$.
\end{lemma}
\begin{proof}
Since the space $C([0,\infty)^2)$ is metrizable, we can use the sequential definition
of compactness, and it suffices to apply the Arzela-Ascoli Theorem.
\end{proof}

\begin{proof}[Proof of Theorem~\ref{thm:main}]
\textbf{Step 1.}
For all $N \in \bN$, let $\phi_N:= \Phi_N(\phi_+, \phi_-): \Omega \times [0, \infty)^2 \to \bR^{d+1}$.
We claim that for all $L \in \bN$, $N \in \{4, 5, \ldots\}$ and $A > 0$ we have
\begin{equation}
\label{eq:cont-mod-phiN}
\begin{aligned}
\bP\bigg(\bigg\{\sup_{\substack{0 \leq u, v, \sh u, \sh v \leq 2^{L} \\ (u, v) \neq (\sh u, \sh v)}}\frac{|\phi_N(\sh u, \sh v) - \phi_N(u, v)|}{h(|(\sh u-u, \sh v - v)|)} \geq 64A\bigg\} \bigg)\leq C4^{L-A^2 / 8}.
\end{aligned}
\end{equation}
Let $Y(m, n) := \phi_N(2^{-N}m, 2^{-N}n)$ for all $0 \leq m, n \leq 2^{L+N}$ and let $\wt Y: [0, 2^{L+N}]^2 \to \bR^{d+1}$
be its extension, so that $\phi_N(u, v) = \wt Y(2^N u, 2^N v)$ for all $0 \leq u, v \leq 2^L$.
It suffices to prove that the event whose probability is taken in \eqref{eq:cont-mod-phiN}
is contained in the event defined in \eqref{eq:cont-mod-ineq}.
Assume that this last event does not hold.

By rescaling $u$ and $v$, we have
\begin{equation}
\sup_{\substack{0 \leq u, v, \sh u, \sh v \leq 2^{L} \\ (u, v) \neq (\sh u, \sh v)}}\frac{|\phi_N(\sh u, \sh v) - \phi_N(u, v)|}{h(|(\sh u-u, \sh v - v)|)} =
\sup_{\substack{0 \leq u, v, \sh u, \sh v \leq 2^{L+N} \\ (u, v) \neq (\sh u, \sh v)}}\frac{|\wt Y(\sh u, \sh v) - \wt Y(u, v)|}{h(2^{-N}|(\sh u-u, \sh v - v)|)}.
\end{equation}
Let $0 \leq u, v, \sh u, \sh v \leq 2^{L+N}$ such that $(u, v) \neq (\sh u, \sh v)$ and let $r := |(\sh u - u, \sh v - v)|$.
If $r \geq 2\sqrt 2$, then \eqref{eq:ext-cont} yields existence of $m, n, \sh m, \sh n$ such that $|\sh m - m| + |\sh n - n| \leq 2r$ and $|\wt Y(\sh u, \sh v) - \wt Y(u, v)| \leq |Y(\sh m, \sh n) - Y(m, n)|$.
Since the event in \eqref{eq:cont-mod-ineq} does not hold, we have
\begin{equation}
|Y(\sh m, \sh n) - Y(m, n)| < 32Ah(2r2^{-N}) \leq 64Ah(2^{-N}r),
\end{equation}
where in the last step we use the inequality $h(2\rho) \leq 2h(\rho)$ for all $\rho > 0$.

Let now $0 < r \leq 2\sqrt 2$. Since the event in \eqref{eq:cont-mod-ineq} does not hold, \eqref{eq:ext-lip} yields
\begin{equation}
\begin{aligned}
|\wt Y(\sh u, \sh v) - \wt Y(u, v)| &\leq \|\partial_u \wt Y\|_{L^\infty}|\sh u - u| + \|\partial_v \wt Y\|_{L^\infty}|\sh v - v| \\
&< 8Ah(2^{-N})(|\sh u - u| + |\sh v - v|) = 8\sqrt 2 Ar h(2^{-N}).
\end{aligned}
\end{equation}
The function $\rho \mapsto \rho^{-1}h(\rho)$ is decreasing, hence
\begin{equation}
\begin{aligned}
2^Nr^{-1}h(2^{-N}r) &\geq 2^N(2\sqrt 2)^{-1}h(2^{-N+1}\sqrt 2) = 2^{\frac N2 - \frac 34}\sqrt{(N-\frac 32)\log 2} \\
& = 2^{-\frac 34}\sqrt{1 - \frac{3}{2N}}2^{\frac N2}\sqrt{N\log 2} \geq \frac 14 2^Nh(2^{-N}),
\end{aligned}
\end{equation}
and we obtain
\begin{equation}
|\wt Y(\sh u, \sh v) - \wt Y(u, v)| \leq 32\sqrt 2 Ah(2^{-N}r) < 64Ah(2^{-N}r).
\end{equation}

\noindent
\textbf{Step 2.}
By Lemma~\ref{lem:pre-compact}, Step 1 and the Prokhorov Theorem, there exists an increasing sequence $(N_k)_k$
and a random field $\phi: \Omega' \to C([0, \infty)^2)$ such that $\phi_{N_k} \to \phi$ in distribution.
For any dyadic $u$ and $v$, we have $\phi_{N_k}(u, v) \in \bS^d$ for all $k$ large enough, hence $\phi(u, v) \in \bS^d$
with probability 1. Since $\phi$ is continuous with probability 1, we deduce that $\phi: \Omega' \to X([0, \infty)^2)$.
Thus, $\phi$ is a solution in the sense of Definition~\ref{def:solution}
(of course, the formulation of the latter is designed precisely in order to make this step trivial).
We verify that $\phi$ has all the desired properties.

Since the Borel $\sigma$-algebra of $C([0, \infty)^2)$ is generated by the cylindrical sets,
see for example \cite[Section I.1.2]{Vakhania+T+Ch_1987},
property \ref{it:invariance} will follow if we can prove that for all $J \in \bN$ and
$(u_j, v_j)_{j=0}^J$, the random sequences
\begin{equation}
\big(\phi(u_0 + u_j, v_0 + v_j)\big)_{j=1}^J, \qquad \big(\phi(u_j, v_j)\big)_{j=1}^J
\end{equation}
have the same laws.
If all $u_j$ and $v_j$ are dyadic, then the claim follows from Corollary~\ref{cor:discr-inv}.
In the general case, let $(u_{j, l})_l$ and $(v_{j, l})_l$ be sequences of dyadic numbers such that
$\lim_{l\to \infty}u_{j, l} = u_j$ and $\lim_{l\to \infty}v_{j, l} = v_j$ for all $j \in \llbracket 0, J\rrbracket$.
Since $\phi$ is continuous with probability 1, for all $j \in \llbracket 1, J\rrbracket$ we have, with probability 1 hence in distribution,
\begin{gather}
\lim_{l \to \infty} \phi(u_{0, l} + u_{j, l}, v_{0, l} + v_{j, l}) = \phi(u_0 + u_j, v_0 + v_j), \\
\lim_{l \to \infty} \phi(u_{j, l}, v_{j, l}) = \phi(u_j, v_j),
\end{gather}
and we already know that the laws of the sequences on the left are the same for every $l$.

Finally, we check \eqref{eq:modulus}. By Step 1, for all $L \in \bN$, $A \geq 0$ and $k$ we have
\begin{equation}
\begin{aligned}
\bP\bigg(\bigg\{\sup_{\substack{0 \leq u, v, \sh u, \sh v \leq 2^{L} \\ (u, v) \neq (\sh u, \sh v)}}\frac{|\phi_{N_k}(\sh u, \sh v) - \phi_{N_k}(u, v)|}{h(|(\sh u-u, \sh v - v)|)} \geq 64A\bigg\} \bigg)\leq C4^{L-A^2 / 8}.
\end{aligned}
\end{equation}
Passing to the limit in distribution, we obtain
\begin{equation}
\begin{aligned}
\bP\bigg(\bigg\{\sup_{\substack{0 \leq u, v, \sh u, \sh v \leq 2^{L} \\ (u, v) \neq (\sh u, \sh v)}}\frac{|\phi(\sh u, \sh v) - \phi(u, v)|}{h(|(\sh u-u, \sh v - v)|)} \geq 64A\bigg\} \bigg)\leq C4^{L-A^2 / 8}.
\end{aligned}
\end{equation}
\end{proof}

\subsection{Link with the Gibbs distribution}
\label{ssec:gibbs}
We finish this section by indicating the relationship between the random field constructed in Theorem~\ref{thm:main},
and the Gibbs distribution associated with the equation, given by \eqref{eq:gibbs}.
As we mentioned in the Introduction, our argument is voluntarily non-rigorous.

Let $(u_1, v_1), (u_2, v_2) \in (0, \infty)^2$ with $u_1 < u_2$ and $v_1 > v_2$.
Letting $(u_1^{(N)}, v_1^{(N)})$ and $(u_2^{(N)}, v_2^{(N)})$ be sequences of pairs of dyadic numbers
with denominator $2^N$, converging to $(u_1, v_1)$ and $(u_2, v_2)$ respectively, we have
\begin{align}
\partial_u\phi(u_1, v_1) &\simeq 2^N\Big(\phi_N\big(u_1^{(N)}, v_1^{(N)}\big) - \phi_N\Big(u_1^{(N)}- \frac{1}{2^N}, v_1^{(N)}\Big)\Big), \\
\partial_v\phi(u_1, v_1) &\simeq 2^N\Big(\phi_N\big(u_1^{(N)}, v_1^{(N)}\big) - \phi_N\Big(u_1^{(N)}, v_1^{(N)}- \frac{1}{2^N}\Big)\Big), \\
\partial_u\phi(u_2, v_2) &\simeq 2^N\Big(\phi_N\big(u_2^{(N)}, v_2^{(N)}\big) - \phi_N\Big(u_2^{(N)}- \frac{1}{2^N}, v_2^{(N)}\Big)\Big), \\
\partial_v\phi(u_2, v_2) &\simeq 2^N\Big(\phi_N\big(u_2^{(N)}, v_2^{(N)}\big) - \phi_N\Big(u_2^{(N)}, v_2^{(N)}- \frac{1}{2^N}\Big)\Big).
\end{align}
By Lemma~\ref{lem:reflections}, the right had sides are pairwise independent and approach standard Gaussians
on the planes tangent to $\bS^d$ at $\phi(u_1, v_1)$ and $\phi(u_2, v_2)$.

We return to the variables $(t, x)$ by setting $\psi(t, x) := \phi(t+x, t-x)$.
Since
\begin{align}
\partial_t \psi(t, x) &= \partial_u \psi(t+x, t-x) + \partial_v \phi(t+x, t-x), \\
\partial_x \psi(t, x) &= \partial_u \psi(t+x, t-x) - \partial_v \phi(t+x, t-x),
\end{align}
we infer that $\partial_t\psi(t, x)$ and $\partial_x \psi(t, x)$ are independent Gaussians with variance $2$
on the plane tangent to $\bS^d$ at $\psi(t, x)$. Moreover, if $(t_1, x_1)$ and $(t_2, x_2)$ are such that
$0 < t_1 + x_1 < t_2 + x_2$ and $t_1 - x_1 > t_2 - x_2 > 0$, then $\partial_t\psi(t_1, x_1)$, $\partial_x \psi(t_1, x_1)$,
$\partial_t\psi(t_2, x_2)$ and $\partial_x \psi(t_2, x_2)$ are pairwise independent.
In particular, this is true for fixed $t_1 = t_2 = T$ and $-T < x_1 < x_2 < T$.

Hence, we obtain that, for fixed $T > 0$,
\begin{equation}
(-T, T) \owns x \mapsto \partial_t\psi(T, x), \qquad (-T, T) \owns x \mapsto\partial_x \psi(T, x)
\end{equation}
are independent white noises, which corresponds to the probability density given by \eqref{eq:gibbs}.

\appendix

\section{Results from Lebesgue Theory}
\begin{lemma}
\label{lem:exist-partial}
Let $u_0, v_0 > 0$, $f \in L^1([0, u_0] \times [0, v_0])$
and let $\phi$ be defined by
\begin{equation}
\phi(u, v) := \int_0^u\int_0^v f(w, z)\ud z\ud w.
\end{equation}
Then there exists a set $A$ of measure $0$ such that $\partial_u\phi(u, v)$
exists for all $(u, v) \in ([0, u_0] \setminus A)\times [0, v_0]$ and
\begin{align}
\partial_u \phi(u, v) = \int_0^v f(u, z)\ud z,
\end{align}
and a set $B$ of measure $0$ such that $\partial_v\phi(u, v)$ exists
for all $(u, v) \in [0, u_0] \times ([0, v_0] \setminus B)$ and
\begin{align}
\partial_v \phi(u, v) = \int_0^u f(w, v)\ud w.
\end{align}
\end{lemma}
\begin{proof}
We give a proof of existence of $A$. Existence of $B$ is obtained analogously.

By the Fubini's Theorem, there exists a set $A_0$ of measure $0$ such that
$f(u, \cdot) \in L^1([0, v_0])$ for all $u \in [0, u_0]\setminus A_0$,
so that for all $(u, v) \in ([0, u_0] \setminus A_0)\times [0, v_0]$ we can define
\begin{equation}
\psi(u, v) := \int_0^v f(u, z)\ud z.
\end{equation}
For all $n \in \bN^*$, let
\begin{equation}
A_n := \bigg\{u \in (0, u_0) \setminus A_0: \limsup_{h \to 0}\sup_{v \in [0, v_0]}
\Big|\frac{\phi(u+h, v) - \phi(u, v)}{h} - \psi(u, v)\Big| \geq \frac 1n\bigg\}.
\end{equation}

We claim that $A_n$ is of measure $0$.
In order to prove this, let $f_m$ be a sequence of smooth functions on $[0, u_0]\times [0, v_0]$
such that for $g_m := f - f_m$ we have $\lim_{m\to\infty}\|g_m\|_{L^1} = 0$, and set
\begin{align}
\phi_m(u, v) &:= \int_0^u\int_0^v f_m(w, z)\ud z\ud w, \\
\psi_m(u, v) &:= \int_0^v f_m(u, z)\ud z.
\end{align}
For all $u\in (0, u_0)$ we have
\begin{equation}
\limsup_{h \to 0}\sup_{v \in [0, v_0]}\Big|\frac{\phi_m(u+h, v) - \phi_m(u, v)}{h} - \psi_m(u, v)\Big| = 0,
\end{equation}
thus for all $u \in A_n$ we obtain
\begin{equation}
\begin{aligned}
\sup_{0 < |h| < \min(u, u_0-u)}\sup_{v \in [0, v_0]}&\Big|\frac{\phi(u+h, v) - \phi(u, v)}{h} - \psi(u, v)\\
&- \frac{\phi_m(u+h, v) - \phi_m(u, v)}{h} + \psi_m(u, v)\Big| \geq \frac 1n.
\end{aligned}
\end{equation}
Now observe that for all $(h, v)$ as above we have
\begin{equation}
\begin{aligned}
\Big|\frac{\phi(u+h, v) - \phi(u, v)}{h} - \frac{\phi_m(u+h, v) - \phi_m(u, v)}{h}\Big|
&= \frac 1h \bigg|\int_u^{u+h}\int_0^v g_m(w, z)\ud z\ud w\bigg| \\
&\qquad \frac 1h\bigg|\int_u^{u+h} G_m(z)\ud z\bigg|,
\end{aligned}
\end{equation}
where
\begin{equation}
G_m(u) := \int_0^{v_0}|g_m(u, z)|\ud z, \qquad \|G_m\|_{L_u^1} = \|g_m\|_{L_{u, v}^1}.
\end{equation}
By the weak $L^1$ bound of the Hardy-Littlewood maximal function,
see for example \cite[Chapter 2.3]{Muscalu-Schlag}, we have $|A_n| \leq \frac{6\|g_m\|_{L^1}}{n}$. Letting $m \to \infty$, we obtain $|A_n| = 0$.

It now suffices to set $A := A_0 \cup A_1 \cup A_2 \cup \ldots$.
\end{proof}

\begin{lemma}
\label{lem:deriv-g+}
Let $u_0, v_0 > 0$, $f \in L^1([0, u_0]\times [0, v_0])$ and let $g_+$ be defined by
\begin{equation}
g_+(u, v) := \int_0^v f(u, z)\ud z
\end{equation}
whenever the last expression is defined.
Then there exists a set $C \subset [0, u_0]\times [0, v_0]$ of measure $0$
such that for all $(u, v) \in [0, u_0]\times [0, v_0] \setminus C$,
$\partial_v g_+(u, v)$ exists and
\begin{equation}
\partial_v g_+(u, v) = f(u, v).
\end{equation}
\end{lemma}
\begin{proof}
Let $A_0$ be defined as in the previous proof, so that for all $u \in [0, u_0] \setminus A_0$,
the function $g_+(u, \cdot)$ is well-defined and continuous.
We define
\begin{equation}
\begin{aligned}
C &:= \big([0, u_0] \times \{0, v_0\}\big) \cup A_0 \times [0, v_0]\  \cup \\
&\cup \Big\{(u, v) \in ([0, u_0]\setminus A_0) \times (0, v_0):
\limsup_{h \to 0}\Big|\frac{g_+(u, v+h) - g_+(u, v)}{h} - f(u, v)\Big| > 0\Big\}.
\end{aligned}
\end{equation}
Then $C$ is a measurable set and, by the Lebesgue Differentiation Theorem,
for all $u \in [0, u_0] \setminus A_0$ the set $C_u := \{v \in (0, v_0): (u, v) \in C\}$
is of measure $0$. By Fubini's Theorem, the set $C$ is of measure $0$.
\end{proof}
\begin{remark}
One can get an analogous result by interchanging the roles of $u$ and $v$.
\end{remark}
\begin{lemma}
\label{lem:preserved-length}
Let $u_0 > 0$, $f \in L^1([0, u_0]; \bR^{d+1})$, $g_0 \in \bR^{d+1}$ and for all $u \in [0, u_0]$
\begin{equation}
\label{eq:g-int-f}
g(u) = g_0 + \int_0^u f(w) \ud w.
\end{equation}
Assume that $g(u)\cdot f(u) = 0$ for almost all $u \in [0, u_0]$. Then
\begin{equation}
\label{eq:preserved-length}
|g(u_0)| = |g_0|.
\end{equation}
\end{lemma}
\begin{proof}
By the Lebesgue Differentiation Theorem, $g'(u) = f(u)$ for almost all $u \in [0, u_0]$.
Thus, for almost all $u \in [0, u_0]$ we have
\begin{equation}
\dd u |g(u)|^2 = 2 g(u)\cdot g'(u) = 2g(u)\cdot f(u) = 0.
\end{equation}
Since $|g|^2$ is absolutely continuous (as a product of two absolutely continuous functions),
we obtain \eqref{eq:preserved-length}.
\end{proof}
\begin{lemma}
\label{lem:L1-lemma}
Let $x_0 > 0$.
For $N \in \bN^*$, let $T_N: L^1([0, x_0)) \to L^1([0, x_0))$ be given by
\begin{equation}
\label{eq:TN-def}
\begin{aligned}
(T_N f)(x) := f(x) - \frac{2^N}{x_0}\int_{\frac{nx_0}{2^{N}}}^{\frac{(n+1)x_0}{2^{N}}}f(t)\ud t,
\end{aligned}
\end{equation}
for all $f \in L^1([0, x_0))$, $n \in \llbracket 0, 2^N-1\rrbracket$ and $x \in \big[\frac{nx_0}{2^{N}}, \frac{(n+1)x_0}{2^{N}}\big)$. Then $\|T_N\|_{\scrL(L^1)} \leq 2$ for all $N$
and $\lim_{N \to \infty}\|T_N f\|_{L^1} = 0$ for all $f \in L^1([0, x_0))$.
\end{lemma}
\begin{proof}
It is clear that $T_N$ is a linear operator and $\|T_N\|_{\scrL(L^1)} \leq 2$.
Thus, it suffices to prove that $T_N f \to 0$ in $L^1([0, x_0))$
for all smooth $f$, which is immediate.
\end{proof}

\providecommand{\noopsort}[1]{}


\begin{thebibliography}{10}

\bibitem{Bourgain-94}
J.~Bourgain.
\newblock Periodic nonlinear {S}chr{\"o}dinger equation and invariant measures.
\newblock {\em Comm. Math. Phys.}, 66:1--26, 1994.

\bibitem{Bringmann}
B.~Bringmann.
\newblock Almost-sure scattering for the radial energy-critical nonlinear wave
  equation in three dimensions.
\newblock {\em Anal. PDE}, 13(4):1011--1050, 2020.

\bibitem{BLS-21}
B.~Bringmann, J.~L{\"u}hrmann, and G.~Staffilani.
\newblock The wave maps equation and {B}rownian paths.
\newblock {\em Preprint, arxiv:2111.07381}.

\bibitem{Brz+Ondr_2007}
Zdzis{\l}aw Brze\'{z}niak and Martin Ondrej\'{a}t.
\newblock Strong solutions to stochastic wave equations with values in
  {R}iemannian manifolds.
\newblock {\em J. Funct. Anal.}, 253(2):449--481, 2007.

\bibitem{Brz+Ondr_2013}
Zdzis{\l}aw Brze\'{z}niak and Martin Ondrej\'{a}t.
\newblock Stochastic geometric wave equations with values in compact
  {R}iemannian homogeneous spaces.
\newblock {\em Ann. Probab.}, 41(3B):1938--1977, 2013.

\bibitem{BTT-18}
N.~Burq, L.~Thomann, and N.~Tzvetkov.
\newblock Remarks on the {G}ibbs measures for nonlinear dispersive equations.
\newblock {\em Ann. Fac. Sci. Toulouse}, XXVII(3):527--597, 2018.

\bibitem{BT-08-1}
N.~Burq and N.~Tzvetkov.
\newblock Random data {C}auchy theory for supercritical wave equations. {I.}
  {L}ocal theory.
\newblock {\em Invent. Math.}, 173(3):449--475, 2008.

\bibitem{BT-08-2}
N.~Burq and N.~Tzvetkov.
\newblock Random data {C}auchy theory for supercritical wave equations. {II}.
  {A} global existence result.
\newblock {\em Invent. Math.}, 173(3):477--496, 2008.

\bibitem{CS-15}
F.~Cacciafesta and A.-S. \noopsort{Suzzoni}{de Suzzoni}.
\newblock Invariant measure for the {S}chr{\"o}dinger equation on the real
  line.
\newblock {\em J. Funct. Anal.}, 269(1):271--324, 2015.

\bibitem{Cheng-Li-Yau}
S.~Y. Cheng, P.~Li, and S.-T. Yau.
\newblock On the upper estimate of the heat kernel of a complete {R}iemannian
  manifold.
\newblock {\em Amer. J. Math.}, 103(5):1021--1063, 1981.

\bibitem{ChBr}
Y.~Choquet-Bruhat.
\newblock Probl{\`e}me des conditions initiales sur un cono{\"i}de
  caract{\'e}ristique.
\newblock {\em C. R. Acad. Sci. Paris}, 256:3971--3973, 1963.

\bibitem{Deng-22}
Y.~Deng, A.~Nahmod, and H.~Yue.
\newblock Random tensors, propagation of randomness, and nonlinear dispersive
  equations.
\newblock {\em Inv. Math.}, 228:539--686, 2022.

\bibitem{GKO-18}
M.~Gubinelli, Koch. H., and T.~Oh.
\newblock Paracontrolled approach to the three-dimensional stochastic nonlinear
  wave equation with quadratic nonlinearity.
\newblock {\em Preprint, to appear in J. Eur. Math. Soc.}, 2018.

\bibitem{Hsu}
E.~P. Hsu.
\newblock {\em Stochastic Analysis on Manifolds}, volume~38 of {\em Graduate
  Studies in Mathematics}.
\newblock AMS, 2002.

\bibitem{Keel+Tao_1998}
M.~Keel and T.~Tao.
\newblock Local and global well-posedness of wave maps on {$\mathbb{R}^{1+1}$}
  for rough data.
\newblock {\em Internat. Math. Res. Notices}, (21):1117--1156, 1998.

\bibitem{KM-19}
C.~Kenig and D.~Mendelson.
\newblock The focusing energy-critical nonlinear wave equation with random
  initial data.
\newblock {\em Int. Math. Res. Not. IMRN}, October 2021(19):14508--14615, 2021.

\bibitem{KMV-20}
R.~Killip, J.~Murphy, and M.~Visan.
\newblock Invariance of white noise for {KdV} on the line.
\newblock {\em Invent. Math.}, 222(1):203--282, 2020.

\bibitem{KLS-20}
J.~Krieger, J.~L{\"u}hrmann, and G.~Staffilani.
\newblock Probabilistic small data global well-posedness of the energy-critical
  {M}axwell-{K}lein-{G}ordon equation.
\newblock {\em Preprint, arXiv:2010.09528}, 2020.

\bibitem{LRS-88}
J.~L. Lebowitz, H.~A. Rose, and E.~R. Speer.
\newblock Statistical mechanics of the nonlinear {S}chr{\"o}dinger equation.
\newblock {\em J. Stat. Phys.}, 50(3/4), 1988.

\bibitem{LM-14}
J.~L{\"u}hrmann and D.~Mendelson.
\newblock Random data {C}auchy theory for nonlinear wave equations of
  power-type on {$\mathbb R^3$}.
\newblock {\em Comm. Partial Differential Equations}, 39(12):2262--2283, 2014.

\bibitem{McKean-SI}
H.~P. McKean.
\newblock {\em Stochastic Integrals}.
\newblock Academic Press, 1969.

\bibitem{McKean-95}
H.~P. McKean.
\newblock Statistical mechanics of nonlinear wave equations. 3. {M}etric
  transitivity for hyperbolic sine-{G}ordon.
\newblock {\em J. Stat. Phys.}, 79(3/4), 1995.

\bibitem{McKean-CMP}
H.~P. McKean.
\newblock Statistical mechanics of nonlinear wave equations (4): Cubic
  {S}chr{\"o}dinger.
\newblock {\em Comm. Math. Phys.}, 168:479--491, 1995.

\bibitem{McKean-Vaninsky-94}
H.~P. McKean and K.~L. Vaninsky.
\newblock Brownian motion with restoring drift: The petit and micro-canonical
  ensembles.
\newblock {\em Comm. Math. Phys.}, 160:615--630, 1994.

\bibitem{Muscalu-Schlag}
C.~Muscalu and W.~Schlag.
\newblock {\em Classical and Multilinear Harmonic Analysis, Volume I}.
\newblock Cambridge University Press, 2013.

\bibitem{ShSt00}
J.~Shatah and M.~Struwe.
\newblock {\em {Geometric Wave Equations}}, volume~2 of {\em Courant Lecture
  Notes in Mathematics}.
\newblock AMS, 2000.

\bibitem{Vakhania+T+Ch_1987}
N.~N. Vakhania, V.~I. Tarieladze, and S.~A. Chobanyan.
\newblock {\em Probability distributions on {B}anach spaces}, volume~14 of {\em
  Mathematics and its Applications (Soviet Series)}.
\newblock D. Reidel Publishing Co., Dordrecht, 1987.
\newblock Translated from the Russian and with a preface by Wojbor A.
  Woyczynski.

\bibitem{Varadhan-1967}
S.~R.~S. Varadhan.
\newblock On the behavior of the fundamental solution of the heat equation with
  variable coefficients.
\newblock {\em Comm. Pure Appl. Math.}, 20:431--455, 1967.

\bibitem{Zhidkov-91}
P.~E. Zhidkov.
\newblock On an invariant measure for a nonlinear {S}chr{\"o}dinger equation.
\newblock {\em Dokl. Akad. Nauk SSSR}, 317(3):543--546, 1991.

\bibitem{Zhidkov-94}
P.~E. Zhidkov.
\newblock An invariant measure for a nonlinear wave equation.
\newblock {\em Nonlinear Anal.}, 22(3):319--325, 1994.

\bibitem{Zhou_1999}
Y.~Zhou.
\newblock Uniqueness of weak solutions of 1+1 dimensional wave maps.
\newblock {\em Math. Z.}, 232(4):707--719, 1999.

\end{thebibliography}
\end{document}